\documentclass{amsart}
\usepackage[dvipsnames]{xcolor}
\usepackage{hyperref}
\hypersetup{
     colorlinks=true,
     linkcolor=blue!60!black,
     filecolor=blue!75!black,
     citecolor = green!50!black,      
     urlcolor=magenta,
     }
\usepackage{ stmaryrd }
\usepackage{amssymb}
\usepackage{ bbold }
\usepackage[all]{xy}
\xyoption{all}
\usepackage{url}
\usepackage{mathrsfs}
\usepackage{tikz-cd}
\usepackage{eucal}
\usepackage{comment}
\usepackage{mathtools}
\usepackage[nameinlink,capitalise,noabbrev]{cleveref}
\usepackage{epigraph} 
\usetikzlibrary{matrix, calc, arrows}

\newcommand{\res}{\operatorname{res}\nolimits}
\newcommand{\Res}{\operatorname{Res}\nolimits}
\newcommand{\Hom}{\operatorname{Hom}\nolimits}
\newcommand{\Spc}{\operatorname{Spc}\nolimits}
\newcommand{\Spec}{\operatorname{Spec}\nolimits}

\def \T{{\mathcal T}}
\def \S{{\mathcal S}}

\def \Fin{\mathcal{F}\textrm{in}}

\def \pp{\mathfrak{p}}
\def \qq{\mathfrak{q}}

\DeclareMathOperator{\supp}{supp}

\DeclareMathOperator{\cosupp}{cosupp}

\DeclareMathOperator{\Mod}{Mod}
\DeclareMathOperator{\mmod}{mod}

\DeclareMathOperator{\CAlg}{CAlg}
\DeclareMathOperator{\Sp}{Sp}

\usepackage{todonotes}

\newcommand{\mmmod}{\mmod\kern-0.1em\text{-}}%
\newcommand{\MMod}{\Mod\kern-0.1em\text{-}}%

\numberwithin{equation}{section}
\newtheorem{theorem}[equation]{Theorem}
\newtheorem{proposition}[equation]{Proposition}
\newtheorem{corollary}[equation]{Corollary}

\theoremstyle{remark}
\newtheorem{definition}[equation]{Definition}

\newtheorem{remark}[equation]{Remark}
\newtheorem{recollection}[equation]{Recollection}
\newtheorem{example}[equation]{Example}

\newtheorem{notation}[equation]{Notation}

\keywords{Chouinard's theorem, weak descendability, conservativity, cohomological stratification}
\subjclass[2020]{55R35;  18F99, 20C99, 55P42, 55P91.}
\thanks{}
\date{\today}

\title{A derived Chouinard{'}s theorem for infinite groups}

\begin{document}

\author[]{Natàlia Castellana}
\author[]{Juan Omar G\'omez}

\address{Natàlia Castellana, Departament de Matemàtiques, Universitat Autònoma de Barcelona, 08193 Bellaterra, Spain; Centre de Recerca Matemàtica, Barcelona, Spain}
\email{Natalia.Castellana@uab.cat}
\urladdr{https://mat.uab.cat/~natalia}

\address{Juan Omar G\'omez, Fakultat f\"ur Mathematik, Universit\"at Bielefeld, D-33501 Bielefeld, Germany}
\email{jgomez@math.uni-bielefeld.de}
\urladdr{https://sites.google.com/cimat.mx/juanomargomez/home}

\begin{abstract}
    We establish a Chouinard-type theorem for modules over algebras of cochains on the classifying spaces of a broad class of infinite groups, allowing commutative ring spectra as coefficients. We then use this result to study finiteness properties of the integral cohomology rings of certain infinite groups, as well as stratification of the corresponding module categories. In particular, we complete the (co)stratification of module categories associated to homotopical groups over fields of positive characteristic.
\end{abstract}

\maketitle

\tableofcontents

\setcounter{tocdepth}{1}
\setlength{\parskip}{0.8ex}


\section{Introduction}

Let $G$ be a finite group and $k$ be a field of positive characteristic $p$ dividing the order of $G$. A fundamental tool in the modular representation theory of $G$ is Chouinard's theorem \cite{Chouinard}, which states that the projectivity of a $kG$–module can be detected on elementary abelian $p$-subgroups. One may reformulate this result as the joint conservativity of the family of restriction functors on stable module categories 
\[
\left(\res^G_E\colon \mathrm{StMod}_{kG} \to \mathrm{StMod}_{kE}\right)_E
\]
where $E$ runs over all elementary abelian $p$-subgroups of $G$.

 As observed by D. Benson and H. Krause in \cite[Section 12]{BK08}, this reformulation of Chouinard's theorem already implies the joint conservativity of the family of restriction functors between homotopy categories of projective $kG$–modules
\[
\left(\res^G_E\colon \mathbf{K}(\mathrm{Proj}(kG)) \to \mathbf{K}(\mathrm{Proj}(kE))\right)_E
\]
 again with $E$ ranging over all elementary abelian $p$-subgroups of $G$. This observation played a central role in the celebrated work of Benson–Iyengar–Krause \cite{BIK11b} on the classification of localizing tensor ideals of $ \mathbf{K}(\mathrm{Proj}(kG))$ and $\mathrm{StMod}_{kG}$, as it allowed them to reduce the problem to the case of elementary abelian groups. 
 
 In a \textit{geometric} context, Chouinard's theorem admits a similar interpretation in the category $\Mod_{C^\ast(BG;\mathbb F_p)}$ of modules over the ring spectrum $C^\ast(BG;\mathbb F_p)$ of $\mathbb F_p$-valued cochains on the classifying space $BG$ of a finite group $G$. In this situation,  Benson-Krause \cite{BK08} proved that the family of induction functors
 \[
 \left(\mathrm{Ind}^E_G\colon \Mod_{C^\ast(BG;\mathbb F_p)} \to \Mod_{C^\ast(BE;\mathbb F_p)}\right)_E
 \]
is jointly conservative, where $E$ ranges over all elementary abelian $p$-subgroups. In fact, in this geometric setting the same statement holds for an arbitrary compact Lie group $G$; see \cite{BG}.

 We generalize the previous result to a broad class of infinite (discrete) groups; see \cref{coro:chouinard in the field case} and \cref{thm-chouinard-mod-p}. 

\begin{theorem}\label{intro-thm-chouinard}
   Let $G$ be a group admitting a finite model for the classifying space for proper actions of $G$, $\underline{E}G$, or a locally finite Artinian group and let $k$ be a field of positive characteristic $p$. Denote by $\mathcal{E}_p(G)$ the collection of elementary abelian $p$-subgroups of $G$. Then the families of induction and coinduction functors
   \[\left(\mathrm{Ind}_{E}^G,\, \mathrm{CoInd}_{G}^E\colon \Mod_{C^*(BG;k)}\xrightarrow[]{}\Mod_{C^*(BE,k)}\right)_{E\in \mathcal{E}_p(G)}\]  are jointly  conservative.  
\end{theorem}

Let us make a few remarks. First, in contrast to the situation for finite groups, induction and coinduction need not agree for infinite groups. Second, the classical Chouinard theorem \cite{Chouinard} holds over more general rings, making it natural to seek a corresponding level of generality in our setting. Finally, our theorem extends to a broader context. Before turning to these developments, we briefly discuss the motivation for generalizing Chouinard's theorem.

Our motivation stems from the work of Barthel--Castellana--Heard--Valenzuela in \cite{BCHV}, which studies the classification of localizing ideals in the module category $\Mod_{C^\ast(B\mathcal{G};\mathbb F_p)}$ via BIK-stratification \cite{BIK}, where $\mathcal{G}$ is a $p$-local compact group. More generally, \cite{BCHV_Noetherian} investigates $\Mod_{C^\ast(X;\mathbb F_p)}$ for spaces $X$ whose mod $p$ cohomology is Noetherian. Both works follow a common descent strategy: establishing a form of strong Quillen stratification of the homogeneous prime ideal spectrum together with a Chouinard-type theorem on $\Mod_{C^\ast(X;\mathbb F_p)}$.

In particular, in the context of \cite{BCHV}, stratification of the module category over the rings considered in \textit{loc. cit.} is shown to be equivalent to  the joint conservativity of the induction and coinduction functors
\[
\left(\mathrm{Ind}_G^E, \, \mathrm{CoInd}^E_G\colon \Mod_{C^\ast(BS;\mathbb F_p)}\to\Mod_{C^\ast(BE;\mathbb F_p)}\right)_{E\in \mathcal{E}_p(G)}
\]
where $S$ is a discrete $p$-toral group. Surprisingly, the question of joint conservativity remained open, and our work provides an answer to this question. It is worth mentioning that in our previous work \cite{Conservative}, we showed that the family of induction functors is jointly conservative for any discrete $p$-toral group $S$, using the theory of purity in triangulated categories.

The primary goal of this work is to establish a Chouinard-type theorem for the category $\Mod_{C^\ast(BG;R)}$ for a broad class of infinite groups, allowing for a more general choice of coefficient ring spectrum $R$. We then deduce consequences for stratification. Our approach is further motivated by the work of  Mathew--Naumann--Noel \cite{mathew2019derived} in equivariant stable homotopy theory, where they introduced the notion of the derived defect base of an equivariant ring. This concept plays an important role in our work, as we now explain.

For a discrete group $G$ and a commutative ring spectrum $R$, we introduce the notion of a $G$-finite derived defect base for $R$. This is a family of finite subgroups of $G$, finite up to conjugacy, such that the cochains over any finite subgroup can be built, in a uniformly bounded number of steps, from cochains associated to groups in the family. See \cref{def-g-finite-derived} for the precise definition.

The following result corresponds to \cref{thm-weakly}. 

\begin{theorem}\label{introthm-weakly}
    Let $G$ be a discrete group with a stable finite-dimensional model for $\underline{E}G$ in the sense of \cite{BDP17}, and $R$ be a commutative ring spectrum. Assume that $R$ has a \textit{$G$-finite derived defect base} $\mathcal{F}$. Then the geometric functor \[\mathrm{Ind}_\mathcal{F}\colon \Mod_{C^\ast(BG;R)}\xrightarrow[]{}\prod_{K\in \mathcal{F}}\Mod_{C^\ast(BK;R)}\] is weakly descendable in the sense of \cref{rec-weakly descendable}. 
\end{theorem}

The upshot is that, if $R$ is of the form $Hk$, where $k$ is a field of positive characteristic $p$, then $R$ admits a $G$-finite derived defect base for locally finite Artinian groups that is contained in the family of elementary abelian $p$-subgroups; see \cref{prop-discrete-ptora-defect}. Moreover, any commutative ring spectrum admits a $G$-finite derived defect base whenever $G$ has a finite model for $\underline{E}G$; see \cref{sec-chouinard} for further details.

As an immediate consequence of the previous theorem, we obtain the desired Chouinard-type result, stated in \cref{intro-thm-chouinard}. In particular, this recovers and extends our earlier result \cite[Theorem 1.5]{Conservative} for locally finite Artinian groups.


Furthermore, this Chouinard-type theorem is not the only direct application of weak descendability. Indeed, by \cite{barthel2024homological}, one obtains a parametrization of the localizing subcategories of $\Mod_{C^\ast(BG;R)}$ in terms of subsets of the so-called homological spectrum; see \cref{sec-homological}. This also completes the work of \cite{BCHV} on the stratification of $p$-local compact groups. The following result corresponds to \cref{thm-costratification-discrete} and \cref{coro-stratification-discrete}.

\begin{theorem}
   Let $\mathcal{G}$ be either a $p$-local compact group with discrete $p$-toral Sylow subgroup $S$, or a locally finite Artinian group. Then $\Mod_{C^*(B\mathcal{G};\mathbb{F}_p)}$ is both stratified and costratified by the canonical action of $H^*(B\mathcal{G};\mathbb{F}_p)$.  
\end{theorem}

We now focus on the category $\Mod_{C^\ast(BG;k)}$ for $G$ a group with a finite model for $\underline{E}G$ and $k$ a commutative Noetherian ring. In this setting, we study when this category is \textit{cohomologically stratified} by the cohomology ring $H^\ast(G;k)$ in the sense of \cite{BIK}. That is, we ask whether the localizing subcategories of $\Mod_{C^\ast(BG;k)}$ are in bijection with subsets of the homogeneous spectrum of $H^\ast(G;k)$ via BIK-support.

The first obstruction to attempting such a cohomological stratification is whether the cohomology ring $H^\ast(G;k)$ is Noetherian. While this question has been addressed in the literature when $k=\mathbb F_p$ (see \cite{hennUnstableModules}, \cite{Qui71}, and \cite{BK02}), little seems to be known in the general case. Using our results on weak descendability, we establish the following theorem, corresponding to \cref{thm-cfg}.

\begin{theorem}\label{introthm-cfg}
      Let $G$ be a group  admitting a finite model for $\underline{E}G$ and $k$ be a commutative Noetherian ring. Then  $H^\ast(G;k)$ is a finitely generated $k$--algebra.  
\end{theorem}

The previous theorem in particular allow us to obtain the following result; see \cref{thm-cohomologica-stra}.

\begin{theorem}
Let $G$ be a group  admitting a finite model for $\underline{E}G$ and $k$ be a commutative Noetherian ring. Then $\Mod_{C^\ast(BG;k)}$ is cohomologically stratified by $H^\ast(G;k)$. 
\end{theorem}

We emphasize that, at this level of generality, this result appears to be new to the best of our knowledge. 

\subsection*{Outline of the document} Section 2 contains preliminaries on module categories over commutative ring spectra, recalls basic facts about classifying spaces for proper actions, and establishes the relevant terminology. In Section 3, we introduce the notion of a $G$-finite derived defect base for a commutative ring spectrum. Section 4 is devoted to the proof of \cref{introthm-weakly} and also provides examples illustrating situations in which the theorem applies. In Section 5, we prove \cref{introthm-cfg}.
Finally, Section 6 is concerned with cohomological stratification of $\Mod_{C^\ast(BG;k)}$ for certain classes of infinite groups and further consequences of \cref{thm-weakly}.

\subsection*{Notation and conventions}

Throughout this paper, we assume familiarity with BIK-stratification \cite{BIK}, as well as with the basics of tensor triangular geometry \cite{Bal10}; see also \cite{stevenson2018tour}. We also adopt the following conventions.

\begin{itemize}
\item We will implicitly view symmetric monoidal stable $\infty$-categories as tensor-triangulated categories via their homotopy categories whenever we invoke results or terminology from tensor triangular geometry, unless the context requires greater precision.

\item For a collection of objects $\mathcal{X}$ in a triangulated category $\T$, we write $\mathrm{Loc}(\mathcal{X})$ and $\mathrm{Thick}(\mathcal{X})$ to denote the smallest localizing subcategory of $\T$ containing $\mathcal{X}$ and the smallest thick subcategory of $\T$ containing $\mathcal{X}$, respectively.

\item We reserve the letter $R$ to denote an arbitrary commutative ring spectrum, and the letter $k$ to denote a discrete commutative ring. We will also occasionally abuse notation by writing $k$ for the Eilenberg–MacLane spectrum associated to $k$.
\end{itemize}

\subsection*{Acknowledgements} 
We are grateful to Drew Heard for his helpful comments on an earlier version of this work.
The first author is deeply grateful to T. Barthel, D. Heard, and G. Valenzuela for many discussions on Chouinard's theorem for discrete $p$-toral groups following our joint work \cite{BCHV}. 

NC would like to thank the Isaac Newton Institute for Mathematical Sciences, Cambridge, for support and hospitality during the programme Equivariant homotopy theory in context where work on this paper was undertaken. This work was supported by EPSRC grant no EP/R014604/1. NC is partially supported by Spanish State Research Agency project PID2024-158573NB-I00, the Severo Ochoa and María de Maeztu Program for Centers and Units of Excellence in R$\&$D (CEX2020-001084-M), and the CERCA Programme of the Generalitat de Catalunya. JOG is supported by the Deutsche Forschungsgemeinschaft (Project-ID 491392403 – TRR 358).


\section{Preliminaries}

This section is devoted to establishing some notation and terminology regarding module categories over ring spectra, classifying spaces for proper actions, and the notion of weak descendability in tensor-triangular geometry. We provide relevant references accordingly.

\subsection{Modules over algebras of cochains} Our main reference for this subsection is \cite[Section 2]{BCHV}.  

\begin{recollection}
    Let $R$ be a commutative ring spectrum. We write $\Mod_R$ to denote the category of $R$-module spectra equipped with the relative smash product $\otimes=\otimes_R$. This is a rigidly-compactly generated triangulated category. Explicitly, this means that $\Mod_R$ is a compactly generated stable $\infty$-category, the functor $\otimes$ is colimit preserving in each variable, the monoidal unit $R$ is a compact object, and any compact object is dualizable. 

In fact, the category $\Mod_R$ is generated by $R$ in the sense that $\Mod_R= \mathrm{Loc}(R)$. A well-known consequence of this fact is that any localizing subcategory of $\Mod_R$ is automatically a $\otimes$-ideal. 

Moreover, since the graded endomorphism ring of the unit in the homotopy category of $\Mod_R$ agrees with $\pi_\ast R$, we obtain that this category is $\pi_\ast R$-linear in the sense of \cite{BIK}.
\end{recollection}

\begin{recollection}\label{rec-base-change}
Let $f\colon R\to S$ be a morphism of commutative ring spectra. Then induction along $f$ (also known as  extension of scalars)
\[
\mathrm{Ind}_f\coloneqq -\otimes_R S\colon \Mod_R\to \Mod_S
\]
is a \textit{geometric functor}; that is, it is a symmetric monoidal exact functor between rigidly-compactly generated tt-categories which commutes with colimits. Hence it comes with a chain of adjunctions 
\[
\mathrm{Ind}_f\dashv \mathrm{Res}_f \dashv \mathrm{Coind}_f
\]
The restriction functor $\mathrm{Res}_f$ is given by restriction along $f$ and it is always conservative. The coinduction functor $\mathrm{Coind}_f$ (also known as coextension of scalars) is given by $\Hom_R(S,-)$. Moreover, the projection formula holds. That is, for any $M\in \Mod_R$ and $N\in \Mod_S$ there is a natural isomorphism
\[
\mathrm{Res}_f( \mathrm{Ind}_fM\otimes  N)\simeq M\otimes \mathrm{Res}_f N.
\]
The projection formula holds for any \textit{geometric functor}; see \cite{BDS16}.
\end{recollection}

We are interested in the following situation. 

\begin{recollection}
 Let $G$ be a (discrete) group and $R$ be a commutative ring spectrum. We write $C^\ast(BG;R)\coloneqq F(\Sigma^\infty_+BG,R)$ to denote the function spectrum of $R$--valued cochains on $BG$, where $BG$ denotes the classifying space of $G$. It is a commutative ring spectrum with the ring structure induced by that of $R$.
  In the particular case where $R=Hk$ is the Eilenberg–MacLane spectrum of a (discrete) commutative ring  $k$, we simply write $C^\ast(BG;k)$. In particular,  
    \[
    \pi_{-i}(C^\ast(BG;k))\cong H^i(G;k).
    \]
 For a subgroup $H \leq G$, we have an induced ring morphism 
 \[
 C^\ast(BG;R)\to C^\ast(BH;R)
 \]
 We write $\mathrm{Ind}_G^H$, $\mathrm{Res}_G^H$ and $\mathrm{Coind}_G^H$ to denote the corresponding induction, restriction and coinduction functors along the previous ring morphism on module categories over $C^\ast(BG;R)$ and $C^\ast(BH;R)$; see \cref{rec-base-change}. 
\end{recollection}

\begin{notation}
   For simplicity, we write $\Hom_G(-,-)\coloneqq \Hom_{\Mod_{C^\ast(BG;R)}}(-,-)$ unless the context might bring some ambiguity.
\end{notation}

We finish this subsection with the following definition. 

\begin{definition}
    We say that a commutative ring spectrum $R$ is \textit{Noetherian} if $\pi_\ast R$ is a (graded) Noetherian ring. 
\end{definition}

\subsection{Classifying spaces for proper actions}\label{sec-classifying} For a group $G$, and a family $\mathcal{F}$ of subgroups of $G$, there is the notion of a classifying space for the family $\mathcal{F}$. We are interested in the case where $\mathcal{F}$ is the family of all finite subgroups of $G$. Let us make this explicit. See \cite{luck2005survey} for a survey on the topic.

\begin{recollection}
 Let $G$ be a group. Then there is a $G$-CW complex $\underline{E}G$ such that the fixed point space $\underline{E}G^H$ is contractible for every finite subgroup $H\subseteq G$, and it is empty for every infinite subgroup $H$. This space is called the \textit{classifying space for proper actions of $G$}; there is no ambiguity in the sense that $\underline{E}G$ is unique up to $G$-homotopy. 

 If there exists a $G$-CW-complex $X\simeq_G \underline{E}G$ with finitely many $G$-equivariant cells, then  we say that $G$ admits a \textit{finite model for $\underline{E}G$}. If there exists a $G$-CW-complex  $X \simeq_G \underline{E}G$ whose $G$-equivariant cells have bounded dimension, then we say that $G$ admits a \textit{finite-dimensional model for $\underline{E}G$}. 
\end{recollection}

In order motivate the relevance of this class of groups, let us provide some explicit examples of groups with a finite-dimensional model for $\underline{E}G$.

\begin{example}\label{ex-locally finite}
Locally finite groups, under certain set-theoretic conditions, provide examples of groups with finite-dimensional models. To make this precise, let $G$ be a non-finitely generated group, and define $\aleph\textrm{-}\mathrm{rank}(G)$ to be the ordinal $\alpha$ such that $\mathrm{rank}(G)=\aleph_\alpha$, where $\mathrm{rank}(G)$ denotes the smallest cardinality of a generating set of $G$. If $G$ is finitely generated, we set $\aleph\textrm{-}\mathrm{rank}(G)\coloneqq -1$.

For a locally finite group $G$, there exists a model for $\underline{E}G$ of dimension $\aleph\textrm{-}\mathrm{rank}(G)+1$; see \cite[Theorem 2.6]{Dicks}. Note that this model is not finite unless $G$ is a finite group, in which case the model is a point. 
\end{example}

    \begin{example}\label{ex-finite vcd}
Groups of finite virtual cohomological dimension over $\mathbb{Z}$ also admit finite-dimensional models for $\underline{E}G$. A standard example is $SL_n(\mathbb{Z})$. See \cite[Theorem VIII.11.1]{Bro82}.
\end{example}

\begin{example}
   By Bass–Serre theory, the fundamental group of a graph of finite groups acts on a tree with finite isotropy. Consequently, such groups admit finite-dimensional models for $\underline{E}G$. See for instance \cite{DD89}.
\end{example}

\begin{example}
    Coxeter groups provide further examples of groups with finite models for $\underline{E}G$; see \cite[Proposition 1.1 and Corollary 1.2]{CharneyDavis}.
\end{example}

\subsection{Finitely many conjugacy classes of elementary abelian subgroups} In this subsection, we discuss some classes of infinite groups with only a finite number of conjugacy classes of elementary abelian subgroups. We refer to \cite{Kro93} for unexplained terminology.

\begin{definition}
  A group $G$ is of \textit{type $FP_\infty$} if the trivial module $\mathbb{Z}$ has a resolution by finitely generated projective $\mathbb{Z}G$--modules.
\end{definition}

\begin{example}
    Let $G$ be a group with a finite model for $\underline{E}G$. Then $G$ is a group of type $\mathrm{FP}_\infty$.
\end{example}

Our interest in this class of groups lies in the following property; see for instance \cite[Theorem 1.3]{Ben97}.

\begin{theorem}\label{thm-benson-abelem}
Let $G$ be a group with a finite-dimensional model for $\underline{E}G$. If $G$ is of type $\mathrm{FP}_\infty$, then there are only finitely many elementary abelian subgroups of $G$ up to conjugacy.
\end{theorem}

\begin{remark}\label{rk:finite model finite conjugacy classes}
Of course, if there exists a finite model for $\underline{E}G$, then $G$ has only finitely many conjugacy classes of finite subgroups, not merely finitely many conjugacy classes of elementary abelian subgroups. 
\end{remark}

\begin{remark}
We stress that the converse of the previous theorem does not hold. For instance, we will see that the conclusion is also valid for locally finite Artinian groups which are not of type $FP_\infty$, as soon as they are not finite. Indeed, this follows since such groups are not even finitely generated.
\end{remark}

\begin{recollection}
    \label{rec: structure l.f. Artinian group}
    Let $G$ be a locally finite Artinian group. By \cite[Recollection 6.5]{Conservative} (see also \cite[Proposition 1.2]{BLO07}), the group $G$ is countable. In particular, it can be written as a filtered union $G=\cup_{n\in \mathbb{N}} G_n$ where each $G_n$ is a finite subgroup.
    
    Moreover, $G$ satisfies strong finiteness conditions with respect to conjugacy classes of certain subgroups. For every positive integer $n>0$, there are only finitely many conjugacy classes of subgroups of order $n$, and these subgroups have finite $p$-rank for every prime $p$ (see the proof of \cite[Lemma 1.4]{BLO07}). In particular, $G$ has only finitely many conjugacy classes of elementary abelian subgroups.
\end{recollection}

\subsection{Weak descendability} In this subsection we recall the notion of weak descendability in tensor-triangular geometry and some relevant consequences that we need later in this work. This notion has been studied in several papers, see e.g. \cite{BCHS}, \cite{barthel2024homological}, and \cite{Gom25}. We refer the reader to these references for further details.

Recall that a \textit{geometric functor} $f^\ast\colon \T\to \S$ is a coproduct-preserving tt-functor between rigidly-compactly generated tt-categories (see \cite{BDS16}). Let us fix a set $I$ for the rest of this section.

\begin{definition}\label{rec-weakly descendable}
      A family of geometric functors  
     \[
     (f_i^*\colon \mathcal{T} \to  \mathcal{S}_i)_{i\in {I}}
     \]
     is weakly descendable if the monoidal unit $\mathbb{1}_\mathcal{T}$ lies in $\mathrm{Locid}\left( (f_i)_*(\mathbb{1}_{\mathcal{S}_i})\mid i\in {I}\right)$, the localizing $\otimes$-ideal   generated by  the set $\{(f_i)_*(\mathbb{1}_{\mathcal{S}_i})\mid i\in {I}\}$. 
\end{definition}

\begin{remark}
Recall that a functor is \textit{conservative} if it reflects isomorphisms. In the stable setting, this is equivalent to detection of the trivial object. A family of exact functors $(f_i^\ast\colon \T\to \S_i)_{i\in {I}}$ is \textit{jointly conservative} if for any $t\in \T$, we have that $t\simeq 0$ if and only if $f_i^\ast (t)\simeq 0$ for all $i\in I$.
\end{remark}

\begin{remark}\label{weakly des are conservative}
If a family $(f^*_i\colon \T\to \S_i)_{i\in I}$ of geometric functors is weakly descendable, then for any $t\in \T$ we have 
\[
t \in \mathrm{Locid}\langle (f_i)_\ast( f_i^\ast(t) )\mid i\in I\rangle \quad \mbox{and} \quad t\in \mathrm{Colocid}\langle (f_i)_\ast( f_i^!(t) )\mid i\in I \rangle.
\]
This follows from the projection formula. In particular,  both families  $(f_i^\ast)_{i\in I}$ and $(f_i^!)_{i\in I}$ are jointly-conservative.  
\end{remark}


\section{Derived defect bases for infinite groups}

In this section, we extend the notion of derived defect base due to Mathew-Naumann-Noel \cite{mathew2019derived} to the context of modules over the cochain algebra of an arbitrary discrete group. We need some preparation. 

\begin{recollection}\label{def:F nilpotent}
   Let $G$ be a finite group and let $\mathcal{F}$ be a family of subgroups of $G$. Let $\Sp_G$ denote the symmetric monoidal $\infty$-category of $G$-spectra as in \cite[Appendix C]{gepner2023equivariant}, which agrees with the underlying $\infty$-category of the category of orthogonal $G$-spectra equipped with the stable model structure \cite[Proposition C.9]{gepner2023equivariant}. For $R \in \CAlg(\Sp_G)$, following \cite{mathew2019derived}, we say that $R$ is \textit{$\mathcal{F}$-nilpotent} if $R$ is in the thick $\otimes$-ideal generated by $A_\mathcal{F}$ which denotes the commutative algebra 
    \[ \prod_{H\in \mathcal{F}} \mathbb{D}(G/H_+)\in \CAlg(\Sp_G),\]  here  $\mathbb{D}(-)$ is short for the internal dual $\mathrm{hom}_{\mathrm{Sp}_G}(-,\mathbb 1)$. By \cite[Proposition 3.26]{BCHNP}, $R$ is \textit{$\mathcal{F}$-nilpotent} if and only if the category of $R$-modules $\mathrm{Mod}_{\Sp_G}(R)$ in $\Sp_G$ is compactly generated by $R\otimes A_\mathcal{F}$, that is,
    \begin{equation}
        \label{eq-loc-ddefectbase}
        \mathrm{Mod}_{\Sp_G}(R)=\mathrm{Loc}(R\otimes A_\mathcal{F})
    \end{equation}
    The minimal family $\mathcal{F}$ of subgroups of $G$ for which $R$ is $\mathcal{F}$-nilpotent is called the \textit{derived defect base} of $R$. 
\end{recollection}

\begin{notation}
For a commutative ring spectrum $R\in \CAlg(\Sp)$, we write $\underline{R}_G$ (or simply $\underline{R}$, when the context is clear) for its Borel completion, that is, $\underline{R}_G=F(\Sigma^\infty_+EG,R)\in \CAlg(\Sp_G)$.
\end{notation}

\begin{recollection}\label{rec-Morita theory for R_G}
   Let $G$ be a finite group and $R$ a commutative ring spectrum. Then the endomorphism ring of the monoidal unit $\underline{R}$ in $\Mod_{\mathrm{Sp}_G}(\underline{R})$ identifies with $C^\ast(BG;R)$ as a commutative ring spectrum. Indeed, the argument of \cite[Lemma 7.8]{BCHNP} applies verbatim in this setting. Moreover, by Morita theory we obtain a fully faithful  tt-functor (see \cite[Theorem 7.1.2.1]{Lur17}) 
    \[f^*\colon \Mod_{C^\ast(BG;R)}\hookrightarrow \Mod_{\mathrm{Sp}_G}(\underline{R})\] 
  whose essential image is the localizing subcategory $\mathrm{Loc}_{\Mod_{\mathrm{Sp}_G}(\underline{R})}(\underline{R})$ generated by the monoidal unit $\underline{R}$. We can identify then $\Mod_{C^\ast(BG;R)}$ with the subcategory of cellular objects in $\Mod_{\mathrm{Sp}_G}(\underline{R})$. Let 
  \[
  f_\ast\colon 
\Mod_{\Sp_G}(\underline{R})\to  \Mod_{C^\ast(BG;R)}
  \]
  be its right adjoint. Then $f_*(\underline{R})\simeq C^*(BG;R)$. 
\end{recollection}

\begin{recollection}\label{rec-compatibility-cochains}
    Let $H\leq G$ be a subgroup. By \cite[Lemma 3.8]{BCHNP},  the restriction functor $\textrm{res}\colon \Mod_{\mathrm{Sp}_G}(\underline{R}) \to \Mod_{\mathrm{Sp}_H}(\underline{R})$ is a finite étale tt-functor with corresponding separable commutative algebra given
by $\mathbb{D}(G/H_+)\otimes \underline{R}$. We also have, a commutative diagram of left adjoint functors 
\begin{center}
    \begin{tikzcd}
       \Mod_{C^\ast(BG;R)} \ar[d,"f^\ast"] \ar[r,"\mathrm{Ind}^H_G"] & \Mod_{C^\ast(BH;R)} \ar[d,"f^\ast"]\\
        \Mod_{\mathrm{Sp}_G}(\underline{R}) \ar[r,"\res"] & \Mod_{\mathrm{Sp}_H}(\underline{R}).
    \end{tikzcd}
\end{center}
It follows that the right adjoints from the bottom right to the top left also commute. Thus we get that
\[f_*(\mathbb{D}(G/H_+)\otimes \underline{R})=\mathrm{Res}^H_G(C^\ast(BH;R)).\]
\end{recollection}

Let us record the following  observation. 

\begin{proposition}\label{prop:mod cochains as localizing}
  Let $G$ be a finite group and $R$ a commutative ring spectrum. Suppose $\mathcal{F}$ is a family of subgroups of $G$ containing the derived defect base $\mathcal{F}(\underline{R})$ of $\underline{R}$. Then 
    \[
    C^\ast(BG;R)\in\mathrm{Thick}\left(\prod_{H\in \mathcal{F}}\mathrm{Res}^H_G(C^\ast(BH;R))\right).
    \]
  In particular, the category $\Mod_{C^\ast(BG;R)}$ is generated, as a localizing subcategory, by
    \[
    \prod_{H\in \mathcal{F}}\mathrm{Res}^H_G(C^\ast(BH;R)).
    \]
\end{proposition}

\begin{proof}
  By assumption, $\underline{R}$ is $\mathcal{F}$-nilpotent, and hence, by \cref{eq-loc-ddefectbase}, $\Mod_{\Sp_G}(\underline{R})$ is compactly generated by $\underline{R}\otimes A_\mathcal{F}$. By \cite[Lemma 2.2]{Nee92b}, we have
   \[
   \Mod_{\Sp_G}(\underline{R})^\omega=\mathrm{Loc}(\underline{R}\otimes A_\mathcal{F})\cap \Mod_{\Sp_G}(\underline{R})^\omega= \mathrm{Thick}(\underline{R}\otimes A_\mathcal{F})
   \]
   and hence we deduce that $\underline{R}$ must lie in $\mathrm{Thick}(\underline{R}\otimes A_\mathcal{F})$. 
   
   Now, let $f_\ast\colon \Mod_{\Sp_G}(\underline{R})\to  \Mod_{C^\ast(BG;R)}$
   be the right adjoint of the fully faithful tt-functor $f^\ast\colon\Mod_{C^\ast(BG;R)} \hookrightarrow \Mod_{\Sp_G}(\underline{R})$ from \cref{rec-Morita theory for R_G}. By the previous observation, we obtain
   \[
   C^\ast(BG;R)\simeq f_\ast(\underline{R})\in \mathrm{Thick}(f_\ast(\underline{R}\otimes A_\mathcal{F})).
   \]
   The result follows from the equivalence $f_\ast(R\otimes A_\mathcal{F})\simeq\prod_{H\in \mathcal{F}}\mathrm{Res}^G_H(C^\ast(BH;R))$; see \cref{rec-compatibility-cochains}. The second claim is an immediate consequence.
\end{proof}

\begin{notation}
Let $G$ be a discrete group. We write $\Fin(G)$ to denote the family of finite subgroups of $G$.   
\end{notation}

\begin{definition}\label{def-g-finite-derived}
    Let $G$ be a discrete group and  let $R$ be a commutative ring spectrum. For each $H\in \Fin(G)$, consider the derived defect base $\mathcal{F}(\underline{R}_H)$ of $\underline{R}_{H}\in \Sp_H$. We say that $R$ has a \textit{$G$-finite derived defect base} if the following two conditions hold.
    \begin{enumerate}
        \item The family \[
    \mathcal{F}(G,R)=\mathcal{F}(R)\coloneqq \bigcup_{H\in \Fin(G)}\mathcal{F}(\underline{R}_H)
    \]
    is finite  up to $G$-conjugacy.
    \item There exists $n>0$ such that, for any finite subgroup $H\leq G$, the object 
    \[
    \prod_{K\in \mathcal{F}(\underline{R}_H)}\mathrm{Res}^K_H(C^\ast(BK;R))\in \Mod_{C^*(BH;R)}
    \]
     generates $C^\ast(BH;R)$, in the sense of \cite{rouquier2008dimensions}, in at most $n$ steps.  
    \end{enumerate}
    In this case, \textit{a $G$-finite derived defect base} of $R$ is a set $\mathcal{F}$ of representatives of the $G$-conjugacy classes of $\mathcal{F}(G,R)$.
\end{definition}

\begin{remark}
Note that condition $(b)$ in the above definition makes sense in view of \cref{prop:mod cochains as localizing}.
\end{remark}

We conclude this section with some examples of rings and groups satisfying the previous definition. 

\begin{example}\label{ex-EG-finite}
    Let $G$ be a group with a finite model for $\underline{E}G$. Then any ring spectrum $R$ has a $G$-finite derived defect base. Indeed, this follows since there are only finitely many finite subgroups of $G$ up to conjugacy. See \cref{rk:finite model finite conjugacy classes}.
\end{example}

We need the following recollection. 

\begin{recollection}
    \label{rec-defect-finitegps}
     Let $G$ be a finite group and $k$ be a commutative Noetherian ring. By 
   \cite[Corollary 6.21]{MNN17}, there is a functor 
   \[
   \Mod_{\mathrm{Sp}_G}(\underline{k}) \to \mathrm{Fun}(BG,\Mod_k) \simeq \mathbf{D}(kG)
   \]
which is fully faithful on compact objects, and which sends the compact object $\underline{k}\otimes \mathbb D(G/H_+)$ to the permutation module $k(G/H)$. 
   
  On the other hand,  a result of Carlson \cite{Car00} shows that the trivial representation $k$ lies in the thick subcategory of $\mathbf{D}^b(kG)$ generated by  the algebra 
    \[\prod_{E\in \mathcal{E}(G)} k(G/E)\]
    where $\mathcal{E}(G)$ denotes the set of elementary abelian subgroups of $G$. See also  \cite{BK08} for the modular case. 

     Combining these results, we deduce that $\underline{k}$ belongs to the subcategory of $\Mod_{\mathrm{Sp}_G}(\underline{k})$ generated by the algebra 
     \[
     \underline{k}\otimes \prod_{E\in \mathcal{E}(G)} \mathbb D(G/E_+)=\underline{k}\otimes A_{\mathcal{E}},
     \]
      and hence the derived defect base of $\underline{k}$ must be contained in $\mathcal{E}(G)$. 
\end{recollection}

    \begin{proposition}
    \label{prop-discrete-ptora-defect}
Let $G$ be a locally finite Artinian group, and let $k$ be a field of characteristic $q$. Then $k$ has a $G$-finite derived defect base which is contained in the class of elementary abelian $q$-subgroups of $G$.
\end{proposition}

\begin{proof}

First, recall that if $G$ is a locally finite Artinian group, then $G$ can be described as an extension of a finite group $P$ by subgroup $H=\prod_p(\mathbb Z/p^{\infty})^{r_p}$ where the product runs over a finite collection of primes $p$ with $r_p> 0$. In particular, we can write $G=\cup_{n\in \mathbb N}G_n$  for an ascending chain of finite subgroups, where each $G_n$ is an extension of $P$ by a finitely generated abelian group of $p$-rank $r_p$ at each $p$; see \cite{BLO07}.

    Now, in view of \cref{rec-defect-finitegps}, we have that $\mathcal{F}(G,k)$ is contained in $\mathcal{E}_q(G)$. Moreover, the latter is finite up to $G$-conjugacy by \cref{rec: structure l.f. Artinian group}. Thus, condition $(a)$ of \cref{def-g-finite-derived} holds.

    To verify condition $(b)$,  note that such a group $G$ admits a monomorphism into a unitary group: first consider the complex faithful representation of $H$  
    \[
    \mu\colon H=\prod_p(\mathbb Z/p^{\infty})^{r_p} \to U\left(\sum_p r_p\right)
    \]
    given by the sum of $\mathbb Z/p^{\infty}\hookrightarrow U(1)$ for each summand $\mathbb Z/p^\infty$ of $H$. Since $H$ is a subgroup of finite index we can apply induction, and the induced representation from $H$ to $G$ gives the desired monomorphism $\rho\colon G\hookrightarrow U(|P|\cdot(\sum r_p))$. In particular, the restriction to each finite subgroup $K\leq G$ produces faithful complex representations $\rho_K\colon K \hookrightarrow U(N)$ where $N=|P|\cdot(\sum r_p)$.
    
    The goal is to show that, the object $C^\ast(BK;k)$ is generated in a number of steps by
     \[
     \prod_{E\in \mathcal{E}_q(K)}C^\ast(BE;k)
     \]
      bounded by some $N>0$ independent of $K$. We claim that $N=|P|\cdot(\sum r_p)$ suffices.

     Indeed, we follow the strategy of \cite[Theorem 3.1]{BG}. Let $K\in \Fin(G)$. We use the previous embeddings $\rho_K\colon K\hookrightarrow U(N)$. Let $T$ denote a maximal torus of $U(N)$, and let $S$ denote the subgroup of elements of order dividing $q$ in $T$, which is a maximal elementary abelian $q$-subgroup of $U(N)$. Then $U(N)/S$ is a finite $K$-complex whose isotropy groups are all elementary abelian $q$-subgroups. The $K$-complex $U(N)/S$ is built from finitely many cells of the form $K/E$, and the argument of \cite[Theorem 3.1]{BG} shows that the gluing of these cells gives a recipe for constructing $C^\ast((U(N)/S)_{hK};k)$ from the corresponding $C^\ast(BE;k)$. The number of steps in this construction depends only on the cellular structure of $U(N)/S$ as a $K$-complex, and hence is bounded in terms of its dimension. Finally, since $H^*(BS;k)$ is a free $H^*(BU(N);k)$-module, an Eilenberg-Moore spectral sequence argument shows that the same holds for the pullback along $B\rho_n$. Thus $H^*((U(N)/S)_{hK};k)$ is a free $H^*(BK;k)$-module. As a consequence, we have that $C^*(BK;k)$ is a retract of $C^\ast((U(N)/S)_{hK};k)$. Although the cellular structure may vary with $K$, its dimension $N$ does not. Thus, the number of steps required to generate $C^\ast(BK;k)$ is uniformly bounded as $K\in \Fin(G)$ varies, proving condition $(b)$ and hence the result. 
\end{proof}

\begin{remark} 
  We note that the preceding strategy for proving condition $(b)$ of \cref{def-g-finite-derived} generalizes to commutative ring spectra $E$ representing complex-oriented cohomology theories. We follow \cite{HKR}. Indeed, given a monomorphism $\rho\colon G\rightarrow U(N)$ with $G$ finite, consider $U(N)/T$, where $T$ is a maximal torus. By \cite[Proposition 2.6(1)]{HKR}, $E^\ast((U(N)/T)_{hG})$ is a free $E^\ast(BG)$-module. Moreover, $C^\ast(BG;E)$ is a retract of $C^\ast((U(N)/T)_{hG};E)$, which is built in finitely many steps from cochains over abelian subgroups of $G$. In particular, the dimension of $U(N)/T$ again provides a bound for the number of steps required when considering the family of abelian subgroups. See also \cite[Theorem 1.5, Remark 5.15]{MNN17}.
\end{remark}

\begin{remark}
On the other hand, there is a broader class of groups for which the collection of elementary abelian subgroups is finite up to conjugacy, including groups of type $FP_\infty$ (see \cref{thm-benson-abelem}) and locally finite Artinian groups (see \cref{rec: structure l.f. Artinian group}). It is natural to ask whether commutative Noetherian rings admit a $G$-finite derived defect base with respect to such a class of groups. 
\end{remark}


\section{Chouinard theorem for groups with finite-dimensional \texorpdfstring{$\underline{E}G$}{underline{E}G}}\label{sec-chouinard}

We now generalize, by different methods, some of the results developed in \cite[Section 6]{Conservative}. We show that groups admitting a finite-dimensional stable model for the classifying space for proper actions satisfy Chouinard's condition, provided that the ground ring spectrum has a finite derived defect base. In fact, we prove a stronger statement: a suitable induction functor is weakly descendable.

A key input is the observation that, whenever the group admits a finite-dimen\-sio\-nal model for $\underline{E}G$, the cochains $C^\ast(BG;R)$ can be reconstructed from the cochains $C^\ast(BH;R)$, where $H$ ranges over a suitable family of finite subgroups of $G$. To illustrate the main ideas, we begin with an example that served as the starting point for our previous work.

\subsection{Groups acting on trees}

Let $G$ be group acting on a tree $T$ with finite isotropy (equivalently, the fundamental group of a graph of finite groups). The edges of $T$ are oriented, so we have a source $\iota(e)$ and target $\tau(e)$ of an edge $e$ in $T$.  Let $Y$ denote a fundamental domain for the action, and write $Y_0$ to denote a maximal subtree of $Y$. In particular, $Y_0$ and $Y$ have the same vertices, but an edge in $Y\backslash Y_0$ has only its source vertex in $Y_0$, while an edge in $ Y_0$ has both its source and target vertices in $Y_0$. 

Now, for each vertex $v$ in $T$, we write $\bar v$ to denote the unique vertex in $Y$ in the same $G$--orbit as $v$. We choose an element $t_v\in G$ such that $t_v \bar v=v$, and let $t_v=1$ in case that $\bar v=v$. This will give us group homomorphisms between isotropies as follows. For a simplex $\sigma$ of $T$, write $G_\sigma$ to denote the isotropy group of $\sigma$. In particular,  we have $G_{\iota(e)}\cap G_{\tau(e)}=G_e$. Hence for any edge $e$ in $Y$, we obtain two maps $f_{\iota(e)}\colon G_e \to G_{\iota(e)}$ given by the inclusion and  $f_{\tau(e)}\colon G_e \to G_{\overline{\tau(e)}}$ given by $g\mapsto t_{\tau(e)}^{-1}gt_{\tau(e)}$.

\begin{recollection}
The augmented chain complex of the tree $T$ with coefficients in $k$ is an exact complex of $kG$--modules of the form 
\[
0\to \bigoplus_{e\in EY} \mathrm{Ind}_{G_e}^G(k) \xrightarrow[]{\tau^\ast-\iota^\ast} \bigoplus_{v\in VY} \mathrm{Ind}_{G_v}^G(k) \to k \to 0
\]
where $VY$ and $EY$ denote the set of vertices and edges of $Y$, respectively. Here the maps $\iota^\ast$ and $\tau^\ast$ are induced by the morphisms $f_{\iota(e)}$ and $f_{\tau(e)}$, respectively.  Using this complex and Frobenius reciprocity, we obtain a long exact sequence on cohomology groups:
\begin{equation}\label{long seq of cohomo}
0\to H^{0}(G;k)\to \prod_{v\in VY } H^0(G_v;k) \to \prod_{e\in EY } H^0(G_e;k) \to H^1(G;k) \to \ldots  
\end{equation}
In fact, this sequence can be interpreted as the collapsing of the Bredon cohomology spectral sequence  which computes the $G$-equivariant cohomology of $T$ (c.f. \cite[Chapter VII]{Bro82}). 
\end{recollection}

\begin{recollection}
 Note that for any edge $e$ of $Y$, we have a commutative diagram 
 \begin{center}
     \begin{tikzcd}
            C^\ast(BG;k)  \arrow[r] \arrow[d] & C^\ast(BG_{\tau(e)};k) \arrow[d] \\
            C^\ast(BG_{\iota(e)};k) \arrow[r] & C^\ast(BG_e;k)
     \end{tikzcd}
 \end{center}
 induced by the group morphisms $G_e\xrightarrow[]{f_\iota(e)}G_{\iota(e)}\to G$ and $G_e\xrightarrow[]{f_\tau(e)}G_{\iota(e)}\to G$. Then we get a comparison map 
\begin{equation}\label{fiber for gps-trees}
     C^\ast(BG;k)\xrightarrow[]{\varphi}\mathrm{fib}\left( \prod_{v\in VY}C^\ast(BG_v;k)\xrightarrow[]{\iota^\ast -\tau^\ast} \prod_{e\in EY}C^\ast(BG_e;k)\right)
 \end{equation}
 where $\iota^\ast$ is induced by the $f_{\iota(e)}$ and $\tau^\ast$ is induced by the $f_{\tau(e)}$. 
\end{recollection}

\begin{proposition}
    The comparison map $\varphi$ from  \cref{fiber for gps-trees}
    is a homotopy equivalence. 
\end{proposition}

\begin{proof}
    Let $M$ denote the right hand side. We will show that $\varphi$ induces an isomorphism in homotopy groups. For this, we use the long exact sequence of homotopy groups associated to a fiber sequence.  In particular, we get that the homotopy groups of $M$ fit in a long exact sequence: 
    \begin{equation}\label{long seq of homotopy groups}
\ldots \to \pi_i M \to \prod_{v\in VY} \pi_i C^\ast(BG_v;k) \to \prod_{e\in VY } \pi_iC^\ast(BG_e;k) \to \ldots  
\end{equation} 
Recall that $\pi_{-i}(C^\ast(BG;k))\cong H^i(G;k)$. In fact, we can interpret this long exact sequence as the collapsing of the Bousfield-Kan spectral sequence for homotopy limits \cite[Section 1.2.2]{Lur17}. Following the same strategy as in the proof of \cite[Theorem 3.1]{gomez2024picard}, we can use the map $\varphi$ to compare the Bredon cohomology spectral sequence with the Bousfield-Kan spectral sequence, and verify that this map indeed is compatible with the filtrations. In other words, the map $\varphi$ induces a map between the long exact sequences from  \cref{long seq of cohomo} and  \cref{long seq of homotopy groups}. We conclude the result by an application of the Five Lemma. 
\end{proof}

\begin{remark}
    The previous result applies, for instance, to locally finite Artinian groups, and recovers \cite[Proposition 6.7]{Conservative}.
\end{remark}

In the following subsection, we will establish an analogous result for groups that admit a \textit{stable} finite-dimensional model for $\underline{E}G$. In particular, we emphasize the role of such models in reconstructing $C^\ast(BG;k)$ from the cochains $C^\ast(BH;k)$, where $H$ ranges over a suitable family of finite subgroups of $G$.

\subsection{Groups with finite-dimensional $\underline{E}G$}

Let $G$ be a discrete group, and write $\Fin(G)$ to denote the family of finite subgroups of $G$.   

\begin{recollection}
   For a discrete group $G$, we denote by $\mathrm{Sp}_G$ the stable $\infty$-category of proper $G$-spectra in the sense of \cite{degrijse2023proper} and \cite{fausk2008equivariant}. A model for $\mathrm{Sp}_G$ is obtained from a stable model category structure on the category of orthogonal $G$-spectra, in which the weak equivalences are those morphisms that induce stable equivalences on derived fixed points for all finite subgroups. We refer to the above references for further details on this category.
\end{recollection}

\begin{definition}
     Following \cite[Section 3]{BDP17}, a \textit{stable model for $\underline{E}G$ }consists of a collection of $G$-spectra $\{X^n\}_{n\in \mathbb Z}$ with  $X^n=\{*\}$ if $n<0$ and cofiber sequences 
    \begin{equation}\label{homotopy cofiber of inclusion}
  X^n\to X^{n+1} \to \bigvee_{i\in I_n} \Sigma^{n+1} (G/H_i)_+ \to \Sigma X^n
\end{equation}
 where $H_i\in \Fin(G)$, and 
 \begin{equation}\label{stable replacement for S0}
     S^0\simeq \varinjlim_{n} X^n
 \end{equation} in the homotopy category of  $\mathrm{Sp}_G$. Moreover, we say that a stable model for $\underline{E}G$ is \textit{finite-dimensional} if there exist $d\geq0$ such that $X^{n}\xrightarrow[]{\simeq}X^{n+1}$ for all $n\geq d$. The smallest $d$ satisfying the previous conditions is called the \textit{dimension} of the stable model. We say that the stable model is \textit{finite} if it is finite-dimensional and $I_n$ is finite for all $n\geq0$. 
\end{definition}

\begin{recollection}\label{rec-G-Topfin}
    Let $G\textrm{-}\mathrm{Top}^{\Fin(G)}_\ast$ denote the category of compactly generated weak Hausdorff spaces with continuous $G$-action and $G$-fixed base point together with $G$-equivariant based continuous maps. We will consider this category equipped with the model structure given by weak equivalences and fibrations those induced on $H$-fixed points, for all $H$ in $\Fin(G)$. For a $G$-space $X$, let $X_+$ denote the space obtained by attaching a disjoint $G$-fixed base-point. Moreover,  there is a suspension spectrum functor $\Sigma^\infty\colon G\textrm{-}\mathrm{Top}^{\Fin(G)}_\ast\to \mathrm{Sp}_G $
which is a left Quillen functor. 
\end{recollection}

The relevance of the suspension spectrum functor is that it allows to construct stable models for $\underline{E}G$ from unstable ones.

 \begin{recollection}
 Let $X$ be a unstable model for  $\underline{E}G$. Its  $G$-CW-decomposition $\varnothing \subset X_0 \subset X_1 \subset \cdots$ gives an stable model by considering $X^n=\Sigma^\infty (X_n)_+$. Indeed, by definition of $\underline{E}G$,  we have 
    \begin{equation}\label{replacement for S0}
        (G/G)_+=S^0\simeq \varinjlim_{n} \; (X_n)_+
    \end{equation}
    in the homotopy category of $G\textrm{-}\mathrm{Top}^{\Fin(G)}_\ast$, where $X_n=\varnothing$ if $n<0$, and there are inclusions of $G$-spaces $(X_n)_+\to (X_{n+1})_+$ which preserve the base-point and fit into a cofiber sequence 
\begin{equation}\label{homotopy cofiber of inclusion +}
  X_+^n\to X_+^{n+1} \to \bigvee_{i\in I_n} \Sigma^{n+1} (G/H_i)_+   
\end{equation}
    with $H_i\in \Fin(G)$, for any $i\in I_n$. Here $\Sigma^n$ denotes the smash product with $S^n$. Moreover, if $X$ has dimension $m<\infty$, we can consider $X^{m}=X^n$ since $X_n\xrightarrow{\simeq} X_{n+1}$, for all $n\geq m$. Since the maps in  \cref{replacement for S0} and \cref{homotopy cofiber of inclusion +} are weak-equivalences of based $G$-CW-complexes, we can translate the analogous statements to the $\infty$-category $\mathrm{Sp}_G$ via the suspension spectrum functor. Moreover, if $X$ is a (finite-dimensional) model for $\underline{E}G$, then $\Sigma^\infty X_+$ is a stable (finite-dimensional) for $\underline{E}G$. Explicitly, the following properties hold:
\begin{enumerate}
\item If $X$ has dimension $d<\infty$, then $\Sigma^\infty X_+$ also has dimension $d$;
\item If $X$ is of finite type, then so is $\Sigma^\infty X_+$.
\end{enumerate}
See \cite[Section 3]{BDP17} for further details.
\end{recollection}

We refer the reader to \cref{sec-classifying} for some examples of groups admitting unstable finite-dimensional models for $\underline{E}G$, and hence  finite-dimensional stable models for $\underline{E}G$.

\begin{proposition}\label{prop:findim model}
Let $G$ be a discrete group with a {stable} finite-dimensional model $X$ for $\underline{E}G$, and $R$ be a commutative ring spectrum. Let 
\[ Y_n= \prod_{H_i \in \Fin(G), \,  i\in I_n} C^\ast(BH_i;R)\in \Mod_{C^\ast(BG;R)}\]
where the $C^\ast(BG;R)$--module structure on $C^\ast(BH_i;R)$ is the one induced by the  inclusion $H_i\subseteq G$, $I_n$ is the set appearing in  \cref{homotopy cofiber of inclusion}, and $H_i$ is the isotropy group of the cell corresponding to $i$. Then $C^\ast(BG;R)$ is in the thick subcategory of $\mathrm{Mod}_{C^\ast(BG;R)}$ generated by the $Y_n$, for $n\in \mathbb{N}$. In symbols,
 \[ C^\ast(BG;R) \in \mathrm{Thick}\left( Y_n\mid n\in \mathbb{N} \right). \] 
\end{proposition}

\begin{proof}
    We write $BG_+$ to denote $\Sigma^\infty_+BG$ for short. From the pointed Borel construction applied to $S^0\simeq \varinjlim_{n} X^n$,  we obtain
    \begin{equation}\label{decomposition for EB+}
        BG_+\simeq \varinjlim_{n\in \mathbb{N}} EG_+\wedge_G X^n.
    \end{equation}
     Recall that $EG_+\wedge_G X^n\simeq S^0$ for $n<0$, and each of the induced maps  
     \[
     EG_+\wedge_GX^{n-1}\to EG_+\wedge_G X^{n}
     \]
     fits in a cofiber sequence 
 \begin{equation}\label{cofiber sequence for the borel construction}
     EG_+\wedge_G X_+^{n-1}\to EG_+\wedge_G X_+^{n} \to \bigvee_{i\in I_n} EG_+\wedge_G \Sigma^n (G/H_i)_+.
 \end{equation}
    Moreover, note that 
    \begin{equation}\label{identification of BH_+}
        EG_+\wedge_G  \Sigma^n(G/H_i)_+\simeq \Sigma^n{BH_i}_+
    \end{equation}
    Now, since $C^\ast(-,R)=F(-,R)$ commutes with colimits, we obtain a homotopy equivalence 
    \begin{equation}\label{C*BG as a limit}
        C^\ast(BG;R)\simeq \varprojlim_{n\in \mathbb{N}}\; C^\ast( EG_+\wedge_G X^n ;R).
    \end{equation}
 We claim that each $C^\ast( EG_+\wedge_G X^n ;R)$ is in $\mathrm{Thick}\left( Y_i\mid i\in \mathbb{N} \right)$. We proceed by induction on $n$. For $n=0$, this is clear since $C^\ast(EG_+\wedge_G X^0;R)\simeq Y_0$ using \cref{identification of BH_+}.  Assume the result for $n-1$, that is, $C^\ast( EG_+\wedge_G X^{n-1} ,R)$ is in $\mathrm{Thick}\left( Y_i\mid i\in \mathbb{N} \right)$. Using the cofiber sequence from \cref{cofiber sequence for the borel construction}, we obtain a fiber sequence 
\begin{equation}\label{Eq:fiber sequence for X^n}
    \Sigma^n Y_n \to C^\ast(EG_+\wedge_G X^n;R) \to C^\ast(EG_+\wedge_G X^{n-1};R)
\end{equation}
where we have identified the left hand side by means of \cref{identification of BH_+}. It follows that both extremes of the above fiber sequence are in $\mathrm{Thick}\left( Y_i\mid i\in \mathbb{N} \right)$, hence so is $C^\ast(EG_+\wedge_G X^n;R)$. Finally, since we are assuming that $X=\{X^n\}$ is a finite-dimensional stable model for $\underline{E}G$, the limit in \cref{C*BG as a limit} is finite. Hence, we conclude that $ C^\ast(BG;R)$ is in $\mathrm{Thick}\left( Y_n\mid n\in \mathbb{N} \right)$.
\end{proof}

We now have all the necessary ingredients to prove our main result.

\begin{theorem}\label{thm-weakly}
      Let $G$ be a discrete group with a stable finite-dimensional model $X$ for $\underline{E}G$, and $R$ be a commutative ring spectrum. Assume that $R$ has a \textit{$G$-finite derived defect base} $\mathcal{F}$. Then the geometric functor \[\mathrm{Ind}_\mathcal{F}\colon \Mod_{C^\ast(BG;R)}\xrightarrow[]{}\prod_{K\in \mathcal{F}}\Mod_{C^\ast(BK;R)}\] is weakly descendable.
\end{theorem}

\begin{proof}
    Recall that  $\mathcal{F}(R)$ denotes  $\cup_{H\in \Fin(G) } \mathcal{F}(\underline{R}_H)$, and that $\mathcal{F}$ is a set of representatives of $G$-conjugacy classes in $\mathcal{F}(R)$. Now, consider the functor 
    \[
    \mathrm{Ind}_{\mathcal{F}(R)} =\prod_{H\in \mathcal{F}(R)} \mathrm{Ind}^{H}_G\colon \Mod_{C^\ast(BG;R)}\xrightarrow[]{}\prod_{H\in \mathcal{F}(R)}\Mod_{C^\ast(BH;R)}
    \]
   which has a right adjoint $\mathrm{Res}_{\mathcal{F}(R)}$. 
    Let $\mathrm{Res}_{\mathcal{F}}$ denote the right adjoint of the functor $\mathrm{Ind}_\mathcal{F}$. We claim that the thick closure of the essential image of $\mathrm{Res}_{\mathcal{F}}$ agrees with the thick closure of the essential image of $\mathrm{Res}_{\mathcal{F}(R)}$. Indeed, for $H\in \Fin(G)$ and  $K\in \mathcal{F}(\underline{R}_H)$, we have that $K$ is conjugate to a subgroup $K'\in \mathcal{F}$ in $G$; that is,  there is an element $g\in G$ such that ${}^gK=K'$ lies in $\mathcal{F}$. In particular, we have a commutative triangle 
\begin{center}
    \begin{tikzcd}
      &  \mathrm{Mod}_{C^\ast(BG;R)}   & \\ 
      \mathrm{Mod}_{C^\ast(BK;R)} \arrow[ur,"\mathrm{Res}_{G}^{K}"]  & & \arrow[ll,"{c^*_g}"'] \mathrm{Mod}_{C^\ast(BK';R)} \arrow[ul,"\mathrm{Res}_G^{K'}"']
    \end{tikzcd} 
\end{center}
where the bottom map is induced by restriction along conjugation by $g$, which is an equivalence. In other words,  $\mathrm{Res}_G^{K} \simeq \mathrm{Res}^{K'}_G \circ {c^\ast_g}$. We deduce that the thick closure of the image of $\mathrm{Res}_G^{K}$ agrees with the thick closure of the image of $\mathrm{Res}_G^{K'}$. Since $\mathrm{Res}_{\mathcal{F}}$ commutes with limits, and $\mathcal{F}(R)$ agrees with with $\mathcal{F}$ up to $G$-conjugacy, we deduce the claim.   

Now,  consider the functor $\mathrm{Res}_{\Fin(G)}= \prod_{H\in \Fin(G)}\mathrm{Res}_G^H$ which is right adjoint to the functor  $\prod_{H\in \Fin(G)}\mathrm{Ind}_G^H$. Since 
\[
\mathrm{Res}_G^H \mathrm{Res}^K_H \simeq \textrm{Res}_G^K 
\]
for any chain of subgroups $K\leq H\leq G$, we deduce that $\mathrm{Res}_{\mathcal{F}(R)}$  
fits into a commutative diagram: 
\begin{center}
    \begin{tikzcd}
       \mathrm{Mod}_{C^\ast(BG;R)} &  &  \prod_{H\in \Fin(G)} \mathrm{Mod}_{C^\ast(BH;R)} \arrow[ll,"\mathrm{Res}_{\Fin(G)}"'] \\ 
       & & \\
   \prod_{K\in \mathcal{F}(R)} \mathrm{Mod}_{C^\ast(BK;R)}\arrow[uu,"\mathrm{Res}_{\mathcal{F}(R)}"]  & & \prod_{H\in \Fin(G)} \mathrm{Mod}_{\prod_{K\in \mathcal{F}(\underline{R}_H)}C^\ast(BK;R)} \arrow[uu,"\mathrm{\prod_{H\in \Fin(G)} Res_{\mathcal{F}(\underline{R}_H)}}"']\arrow[ll]
    \end{tikzcd} 
\end{center}
Consider $\mathcal{C}\subset \prod_{H\in \Fin(G)} \mathrm{Mod}_{C^\ast(BH;R)}$ the thick subcategory generated by the essential image of the right vertical map. Given $n>0$, $\mathcal{C}$ contains all elements $(M_H)_{H\in \Fin(G)}$  such that, for every $H$, the $C^*(BH;R)$-module $M_H$ is generated in at most $n$ steps by elements in the essential image of $\Res_{\mathcal{F}(\underline{R}_H)}$. Since $\mathcal{F}$ is a $G$-finite derived defect base for $R$, condition $(b)$ in \cref{def-g-finite-derived} implies that there exists such $n>0$ when $M_H=C^\ast(BH;R)$. Therefore, we can conclude that
\[
\prod_{H\in \Fin(G)}C^\ast(BH;R)
\]
belongs to $\mathcal{C}$.

Using the commutativity of the above square, we deduce that the module
\begin{equation}\label{eq-prod}
    Y_n=\prod_{H_i \in \Fin(G), \,  i\in I_n} \textrm{Res}_G^{H_i}C^\ast(BH_i,R) 
\end{equation}
also belongs to the thick subcategory generated by the essential image of $\mathrm{Res}_{\mathcal{F}(R)}$. Indeed, it is clear that the module in \cref{eq-prod} belongs to the image of $\mathrm{Res}_{\Fin(G)}$, since the latter commutes with limits. Hence, by the preceding observation, the module in \cref{eq-prod} also belongs to the thick subcategory generated by the essential image of $\Res\mathcal{F}$.

 Thus, the last step is to apply \cref{prop:findim model} to conclude that
\begin{equation}\label{eq-thick}
    {C^\ast(BG;R)} \in \mathrm{Thick}(\mathrm{Im}(\mathrm{Res}_\mathcal{F})) \subseteq \mathrm{Loc}\left(\prod_{K\in \mathcal{F}} \mathrm{Res}^K_G C^\ast(BK;R) \right)
\end{equation}
where the inclusion follows from the fact that $\mathrm{Res}_\mathcal{F}$ commutes with colimits since  $\mathrm{Ind}_\mathcal{F}$ is a geometric functor. Hence $\textrm{Ind}_\mathcal{F}$ is weakly descendable as desired.
\end{proof}

\begin{corollary}\label{coro:chouinard in the field case}
     Let $G$ be a discrete group with a  finite-dimensional stable model $X$ for $\underline{E}G$, and $R$ be a commutative ring spectrum. Assume that $R$ has a \textit{$G$-finite derived defect base} $\mathcal{F}$. Consider  induction and coinduction   
     \[\mathrm{Ind}_{\mathcal{F}},\, \mathrm{CoInd}_\mathcal{F}\colon \Mod_{C^*(BG;R)}\xrightarrow[]{}\Mod_{\prod_{H\in \mathcal{F}}C^*(BH,R)}\] 
     along the ring map $\varphi\colon C^\ast(BG;R)\to \prod_{H\in \mathcal{F}}C^\ast(BH;R)$. Then both functors $\mathrm{Ind}_\mathcal{F}$ and $\mathrm{CoInd}_\mathcal{F}$ are conservative.
\end{corollary}

\begin{remark}
We refer the reader to \cite[Figure 1.7]{mathew2019derived} for a list of ring spectra together with their defect bases; this list can be used to produce concrete examples of ring spectra with a $G$-finite derived defect base, and hence where the previous results apply. Notice that if the group $G$ admits a finite model for $\underline{E}G$, then every ring spectrum automatically has a $G$-finite derived defect base. 
\end{remark}

We conclude this section with some examples of commutative ring spectra that do and do not satisfy condition $(a)$ in \cref{def-g-finite-derived} for locally finite Artinian groups. At present, we do not know whether these ring spectra satisfy condition $(b)$ as well.

\begin{example}
Let $G$ be a locally finite Artinian group. If $R$ is one of $MO$, $K(n)$, or $T(n)$, then $R$ satisfies condition $(a)$ in \cref{def-g-finite-derived}.

In contrast, if $R=KU$, then $R$ does not have a $G$-finite derived defect base whenever $G$ is infinite. Indeed, for a finite group $H$, the defect base of $\underline{KU}_H$ is the family of all cyclic subgroups of $H$, with no restriction on their cardinality. Since, when $G$ is infinite, there are infinitely many $G$-conjugacy classes of cyclic subgroups, the claim follows.
\end{example}


\section{Finiteness properties of the cohomology ring}

Let $G$ be a group admitting a finite-dimensional model for $\underline{E}G$, and let $k$ be a commutative Noetherian ring. In this section, we study finiteness properties of the cohomology algebra $H^\ast(G;k)$.

In particular, for groups that admit a finite model for $\underline{E}G$, we prove that $H^\ast(G;k)$ is finitely generated as a $k$-algebra. For locally finite Artinian groups, the algebra $H^\ast(G;k)$ is finitely generated when $k=\mathbb{F}_p$. We do not know whether the latter result remains true for an arbitrary Noetherian base ring, although we expect this to be the case.

\subsection{Mod $p$ cohomology}
We begin with a finiteness result for the mod $p$ cohomology of a locally finite Artinian group.

\begin{proposition}\label{prop:cohomology locally finite Artinian}
    Let $G$ be a locally finite Artinian group. Then $H^*(G;\mathbb{F}_p)$ is a finitely generated $\mathbb{F}_p$-algebra.
\end{proposition}

\begin{proof}
    Any locally finite Artinian group fits into an extension 
    \[
    1\to A\to G\stackrel{\pi}{\to} F \to 1
    \] 
    where $A\cong \oplus_q (\mathbb{Z}/q^\infty)^{r_q}$ and the sum runs over a finite collection of primes $q$ with $r_q> 0$ and $F$ is a finite group; see \cref{rec: structure l.f. Artinian group}. Let $P\leq F$ be a Sylow $p$-subgroup, then $G'=\pi^{-1}(P)\leq G$ is a subgroup of finite index prime to $p$. By a transfer argument (see \cite[Lemma 2.6]{DW94}), it is enough to show that $H^*(G';\mathbb{F}_p)$ is Noetherian. Note that $G'$ is also a locally finite Artinian group which is an extension of $P$ by $A$.  Now, consider the Hochschild–Serre spectral sequence 
    \[
      E^{s,t}_2= H^s(P,H^t(A;\mathbb{F}_p))   \Rightarrow H^{s+t}(G';\mathbb{F}_p).
    \]
  Now, $A=A_1\times A_2$ with 
\[
 A_1= (\mathbb Z/p^\infty)^{r_p} \mbox{ and } A_2= \left(  \bigoplus_{q\not= p} (\mathbb{Z}/q^\infty)^{r_q} \right).
\]
By the Künneth formula, we get 
\[
H^\ast(A;\mathbb F_p)=H^\ast(A_1\times A_2;\mathbb F_p)\cong H^\ast(A_1;\mathbb F_p)\otimes H^\ast(A_2;\mathbb F_p)\cong H^\ast(A_1;\mathbb F_p)
\]
which is a finitely generated $\mathbb F_p$-algebra as shown in the appendix of \cite{DW94}. By the cohomological finite generation property for finite groups \cite{Eve61}, we obtain that $H^\ast(P;H^\ast(A;\mathbb F_p))$ is a finitely generated $\mathbb F_p$-algebra and a Noetherian module over $H^\ast(A;\mathbb F_p)^P$. The latter implies that $E^{\ast,\ast}_r$ collapses at a finite stage by standard methods \cite{Eve61}, and the former that $E^{\ast,\ast}_2$ is a finitely generated $\mathbb F_p$-algebra, and hence so is $E^{\ast,\ast}_r$ for any $r\geq 2$. It follows then that $H^\ast(G';\mathbb F_p)$ is a finitely generated $\mathbb F_p$-algebra as desired.
\end{proof}

Let us also record the following result, which is a special case of \cite[Theorem 4.8]{BK02}. In the next subsection, we will generalize this to an arbitrary Noetherian base; our approach is not based on this mod-$p$ version. 

\begin{proposition}\label{prop-cohomology finite EG}
     Let $G$ be a group with a finite model for $\underline{E}G$. Then $H^*(BG;\mathbb{F}_p)$ is a finitely generated $\mathbb{F}_p$-algebra.
\end{proposition}

\subsection{Finiteness properties of group cohomology over an arbitrary Noetherian base} 

This subsection is devoted to proving the following result for groups admitting a finite model for the classifying space for proper actions. Our strategy consists of leveraging \cref{thm-weakly}. Along the way, we establish several relevant results that will be used in the subsequent section.

\begin{theorem}\label{thm-cfg}
     Let $G$ be a group admitting a finite model for $\underline{E}G$ and $k$ be a commutative Noetherian ring. Then $H^\ast(G;k)$ is a finitely generated $k$--algebra.
\end{theorem}

\begin{remark}
   Although perhaps surprising, to the best of our knowledge, this result has not previously appeared in the literature in this generality. 
\end{remark}

\begin{recollection}\label{rec-conjugacy-abelem}
   Recall that if $G$ is a group admitting a finite model for $\underline{E}G$, then there are only finitely many $G$-conjugacy classes of elementary abelian subgroups. In fact, there are only finitely many $G$-conjugacy classes of finite subgroups.
\end{recollection}

Let us begin by establishing some terminology.

\begin{definition}
    A morphism of (graded) commutative rings $\alpha\colon A\to B$ is an \textit{$\mathcal{N}$-isomorphism} if every element of $\mathrm{Ker}(\alpha)$ is nilpotent and $\alpha$ is \textit{power surjective}, that is, for every $b\in B$, there exists $n>0$ such that $b^n\in \mathrm{Im}(\alpha)$. We say that $\alpha$ is \textit{uniformly power surjective} if there exists $n>0$ such that $b^n\in \mathrm{Im}(\alpha)$ for every $b\in B$. Finally, we say that $\alpha$ is a \textit{uniform $\mathcal{N}$-isomorphism} if it is uniformly power surjective and there exists $m>0$ such that $x^m=0$ for every $x\in \mathrm{Ker}(\alpha)$. 
\end{definition}

The following result corresponds to \cite[Proposition 3.24]{mathew2019derived}.

\begin{proposition}
    \label{f-iso induces homeo}
    Let $\alpha\colon A\to B$ be an $\mathcal{N}$-isomorphism of  graded commutative rings. Then $\alpha$ induces a homeomorphism on homogeneous spectra 
    \[
    \Spec^h(\alpha) \colon \Spec^h(B) \to \Spec^h(A).
    \]
\end{proposition}

A relevant source of $\mathcal N$-isomorphisms arises from the edge morphism of a multiplicative spectral sequence under suitable vanishing conditions. In particular, we obtain an $\mathcal N$-isomorphism via the equivariant Atiyah–Hirzebruch spectral sequence developed in \cite{DL98} (see also \cite[Section 3.2]{degrijse2023proper}). For the reader’s convenience, we briefly recall this spectral sequence below.

\begin{theorem}\label{thm-Atiyah-Hirzebruch}
    Let $G$ be a group with a finite-dimensional model for $\underline{E}G$. Let $\mathcal{H}^\ast$ be a multiplicative $G$-equivariant cohomology theory, and $X$ be a proper CW-complex. Then there is a conditionally convergent half plane multiplicative spectral sequence  
    \[
    E^{p,q}_2= H_{G}^p(X,\underline{\mathcal{H}}^q)\Rightarrow \mathcal{H}^{p+q}(X)
    \]
    where $H_G^\ast(-)$ denotes Bredon cohomology, and   $\underline{\mathcal{H}}$ denotes the functor given by $\mathcal{H}^q\colon G/H\mapsto \mathcal{H}^q(G/H)$. The $d_r$-differential has bidegree $(r,r-1)$.
\end{theorem}

As a consequence, we obtain the following result. Before stating it, we introduce some notation. 

\begin{notation}
Let $G$ be a group and let $\mathcal{F}$ be a family of subgroups of $G$. We write $\mathcal{O}_\mathcal{F}(G)$ for the orbit category whose objects are $G$-sets of the form $G/H$ for $H\in\mathcal{F}$, and whose morphisms are $G$-equivariant maps. For families such as $\Fin(G)$, we simply write $\mathcal{O}_{\Fin}(G)$ and $\mathcal{O}_{\mathcal{E}}(G)$ to avoid cumbersome notation.
\end{notation}

\begin{proposition}\label{prop-F-i}
    Let $k$ be a commutative ring and $G$ be a group with a finite model for $\underline{E}G$. Then  restriction 
    \[
    H^*(G;k)\to \lim_{\mathcal{O}_{\Fin}(G)^{\mathrm{op}}} H^*(F;k)
    \]
    is an $\mathcal{N}$-isomorphism.
\end{proposition}

\begin{proof}
Recall that  the Atiyah-Hirzebruch spectral sequence for cohomology  is a multiplicative spectral sequence; see \cref{thm-Atiyah-Hirzebruch}. Moreover, in our case it is bounded since we are assuming that $G$ has a finite-dimensional model, that is, $E^{i,j}_2=0$ for $i>\mathrm{dim}(\underline{E}G)$. The edge morphism $H^\ast(G;k)\rightarrow H_{G}^0(X,\underline{\mathcal{H}}^\ast)$ is then an $\mathcal{N}$-isomorphism; see the arguments in \cite[Theorem 3.21]{mathew2019derived}.
Note that 
$H_{G}^0(X,\underline{\mathcal{H}}^\ast)$ can be identified with the limit $\lim_{\mathcal{O}_{\Fin}(G)^{\mathrm{op}}} H^*(F;k)$ by \cite[Section 3.1]{mathew2019derived}; and $H_{G}^i(X,\underline{\mathcal{H}}^\ast)$ are torsion abelian groups for $i>0$ (one can take the least common multiple of the orders of $F\in \Fin(G)$). Indeed, the $G$-space $\underline{E}G$  is simply $E_{\Fin(G)}$ in their notation. 
\end{proof}

\begin{proposition}\label{prop-fg-limitF}
       Let $k$ be a commutative Noetherian ring  and $G$ be a group with a finite model for $\underline{E}G$. Then 
       \[ {B}\coloneqq\lim_{\mathcal{O}_{\Fin}(G)^{\mathrm{op}}} H^\ast(F;k)
       \]
       is a finitely generated $k$--algebra, and $H^\ast(K;k)$ is a finite ${B}$-module via the projection 
       \[
       {B}\subseteq \prod_{F\in \Fin(G)/G} H^*(F;k)\to H^\ast(K;k)
       \]
       for any $K\in \Fin(G)$. 
\end{proposition}

\begin{proof}
 Let $N=\textrm{lcm}(|F| \mid [F]\in \Fin(G)/G)$. For each $F\in \Fin(G)$, consider the faithful complex representation $\rho_F\colon F\hookrightarrow U(N)$ given by 
 \[
 \rho_F\coloneqq\bigoplus_{\frac{N}{|F|}}\rho_{reg}
 \]
 where $\rho_{reg}$ is the complex regular representation of $F$, $\rho_{reg}\colon F\hookrightarrow U(|F|)$. That is, induced by a free $F$-action on a disjoint union of $\frac{N}{|F|}$ copies of $F$. Since $\rho_F$ is a subgroup inclusion, the induced map on cohomology 
 \[
 \rho_F^*\colon H^*(BU(N);k)\rightarrow H^*(F;k)
 \]
 exhibits $H^*(F;k)$ as a finitely generated  $H^*(BU(N);k)$-module, that is, $\rho_F^\ast$ is a Noetherian map\footnote{A map of rings $f\colon A\to B$ is Noetherian if $B$ is a Noetherian $A$-module along $f$.}. Furthermore, the restriction and conjugation of a free action is again free, therefore the corresponding representations are conjugate. We conclude that all these morphisms are compatible with morphisms induced from $\mathcal{O}_{\Fin}(G)$; and all $\rho_F^*$ assemble to give a morphism
\[
\rho^*\colon H^*(BU(N);k)\rightarrow B\subseteq  \prod_{F\in \Fin(G)/G} H^*(F;k).
\]
Recall that $H^*(BU(N);k)$ is a finitely generated (Noetherian) $k$-algebra. Since $\prod_{F\in \Fin(G)/G} H^*(F;k)$ is finitely generated over $H^*(BU(N);k)$, so is ${B}$. We conclude that ${B}$ must be a finitely generated $k$-algebra, and the projection ${B}\to H^\ast(F;k)$ is a Noetherian map since both $\rho^*$ and $\rho_F^*$ are Noetherian.
\end{proof}

\begin{remark}
    Note that the conclusion of \cref{prop-fg-limitF} holds more generally. For an arbitrary group $G$, we may consider a family of subgroups $\mathcal{F}$ with finitely many objects up to $G$-conjugacy. The same arguments then yield a finitely generated $k$-algebra $B$. For example, this applies to locally finite Artinian groups with respect to the family of elementary abelian subgroups.
\end{remark}

We now have all the ingredients we need to prove the main result of this section.  

\begin{proof}[Proof of \cref{thm-cfg}]
   By \cref{prop-F-i} we have that the restriction map 
    \[
    \varphi\colon H^\ast(G;k) \to \lim_{\mathcal{O}_{\Fin}(G)^{\mathrm{op}}} H^\ast(F;k)\eqqcolon {B}
    \]
    is an $\mathcal{N}$-isomorphism. Moreover, \cref{prop-fg-limitF} gives us that ${B}$ is a finitely generated $k$-algebra. Note that $\mathrm{im}(\varphi)$ is a finitely generated $k$-algebra as well. Indeed, this follows by the Artin-Tate Lemma; since $\varphi$ is power surjective, it follows that it is an integral extension (see \cite[Lemma 10.13]{gomez2025notes}), that is, $B$ is integral over $\mathrm{im}(\varphi)$. The claim follows by  \cite[Lemma 10.15]{gomez2025notes}.  Now, chose a finite set of generators $\{x_1,\ldots, x_n\}$ of $\mathrm{im}(\varphi)$. Let $y_1,\ldots,y_n$ in $ H^\ast(BG;k) $ such that $\varphi(y_i)=x_i$, for $1\leq i\leq n$. Then the subalgebra ${A}$ of $H^\ast(BG;k)$ spanned by the  $y_1,\ldots, y_n$ is also a finitely generated $k$--algebra. In fact, note that the map $A\subseteq H^\ast(BG;k)\xrightarrow[]{\varphi} B$ is an $\mathcal{N}$-isomorphism as well. 
    
    Now, consider the module category $\Mod_{C^\ast(BG;k)}$. Recall that the graded endomorphisms ring of the monoidal unit in this category is the cohomology ring $H^\ast(BG;k)$. In particular, $\Mod_{C^\ast(BG;k)}$ is a $H^\ast(BG;k)$--linear category. Hence we can also consider $\Mod_{C^\ast(BG;k)}$ as a ${A}$--linear category via the inclusion ${A}\subseteq H^\ast(BG;k)$. We claim that if $M$ is in 
\[
\mathrm{Thick}\left(\mathrm{Res}^{F}_G(C^\ast(BF;k))\mid F\in \Fin(G)  \right)
\]
then 
$\pi_\ast \Hom_G(C^\ast(BG;k),M)$
is a finitely generated ${A}$--module. Note that this property is a thick property. Hence it is enough to check it for $\mathrm{Res}^{F}_G(C^\ast(BF;k))$. Indeed, 
 we have  
    \[
    \Hom_G(C^\ast(BG;k),\mathrm{Res}^{F}_G(C^\ast(BF;k)))\simeq \Hom_F(C^\ast(BF;k),C^\ast(BF;k)) 
    \]
and the homotopy groups of the latter are $H^*(F;k)$. Then, it is a finitely generated ${A}$--module since the action of $A$ factors through $B$ and the  projection 
$
{B}\to H^\ast(F;k)
$
is a Noetherian map; see \cref{prop-fg-limitF}. Hence the claim follows. 

We conclude by \cref{prop:findim model}, and \cref{eq-thick}, since it gives us that $C^\ast(BG;k)$ is in the thick subcategory of  $\Mod_{C^\ast(BG;k)}$ generated by the modules $\mathrm{Res}^{F}_G(C^\ast(BF;k))$ for $F\in \Fin(G)$.
\end{proof}

\begin{remark}
   Note that the same arguments in the proof of \cref{thm-cfg} could be adapted to a more general ring spectrum $R$ and a group $G$ admitting a finite model for $\underline{E}G$ satisfying the following properties: 
    \begin{enumerate} 
       \item $\pi_\ast C^\ast(BH;R)$ has a $k$-algebra structure for a fixed commutative Noetherian ring $k$ for any $H\in \Fin(G)$;
        \item $\pi_\ast C^\ast(BH;R)$ is finitely generated $k$-algebra for any $H\in \Fin(G)$; 
        \item for any inclusion $H\leq K$, the restriction $\pi_\ast C^\ast(BK;R) \to \pi_\ast C^\ast(BH;R)$ is a Noetherian map.
    \end{enumerate}
      For instance, this is the case if $R = E_n$ is the Lubin
      Tate theory of height $n$ and prime $p$; see \cite[Lemma 7.1]{BCHNP}.
\end{remark}


\section{Stratification of the module category}\label{sec-cohomological}
We now bring together the results from the previous sections - which may until now have seemed somewhat disjoint - to address the classification of localizing ideals of the module category in terms of the homogeneous spectrum of the  cohomology. In a nutshell, both \cref{thm-weakly} and \cref{thm-cfg} are the key ingredients for applying descent techniques in the study of stratification: \cref{thm-cfg} provides control over $\Spec^h(H^*(BG;k))$ via the cohomology of subgroups, while \cref{thm-weakly} allows for the reconstruction of modules via conservative functors. This approach is motivated by ideas developed in \cite{BCHV}.

Moreover, in the sequel we discuss further tt-geometric consequences of our weak descendability result \cref{thm-weakly}, such as homological stratification and the Noetherianity of the Balmer spectrum.

We assume familiarity with BIK stratification, support, and cosupport theory; see \cite{BIK,BIKcolocalizing}.


\subsection{Cohomological stratification mod $p$} We begin by making explicit the following consequence of \cref{thm-weakly}, which, together with \cite{BCHS}, yields cohomological stratification for all $p$-local compact groups.

\begin{theorem}\label{thm-chouinard-mod-p}
    Let $G$ be locally finite Artinian group and $k$ be a field of positive characteristic $p$. Then both families
    \[
    \left(\mathrm{Ind}_G^E, \mathrm{CoInd}_G^E\colon \Mod_{C^\ast(BG;k)})\to \Mod_{C^\ast(BE;k)}\right)_{E\in \mathcal{E}_p(G)}
    \] 
    are jointly conservative. 
\end{theorem}

\begin{proof}
   This follows from \cref{coro:chouinard in the field case}, since $k$ has a $G$-finite derived defect base contained in $\mathcal{E}_p(G)$ by \cref{prop-discrete-ptora-defect}.
\end{proof}

\begin{theorem}\label{thm-costratification-discrete}
Let $\mathcal{G}$ be a $p$-local compact group with a discrete $p$-toral Sylow subgroup $S$. Then $\Mod_{C^*(B\mathcal{G};\mathbb{F}_p)}$ is both stratified and costratified by the canonical action of $H^*(B\mathcal{G};\mathbb{F}_p)$.    
\end{theorem}

\begin{proof}
   The proofs of \cite[Theorems 5.12 and 5.14]{BCHV} carry over to our setting. Indeed, the assumption that $S$ is \textit{geometric} is used to ensure that $S$ satisfies Chouinard's condition. However, by \cref{thm-chouinard-mod-p}, Chouinard's condition holds for an arbitrary locally finite Artinian $p$-group, allowing us to remove the geometricity assumption on $S$.
\end{proof}

Similarly, we can remove the geometricity assumption on $S$ from \cite[Proposition 5.13]{BCHV}. 

\begin{corollary}\label{coro-stratification-discrete}
    Let $S$ be a locally finite Artinian $p$-group. Then $\Mod_{C^*(BS;\mathbb{F}_p)}$ is both stratified and costratified by the canonical action of $H^*(BS;\mathbb{F}_p)$. 
\end{corollary}


\subsection{Cohomological stratification} In this subsection, we show that the module category $\Mod_{C^\ast(BG;k)}$ is BIK-stratified for groups with a finite model for $\underline{E}G$ and $k$ a commutative Noetherian ring.

 \begin{notation}
     Let $f\colon R\to S$ be a map of commutative ring spectra. We write $\varphi_f$ to denote  the induced map on homogeneous prime ideal spectra
     \[
     \varphi_{f}\coloneqq \Spec^h(f)\colon \Spec^h(\pi_*(S))\xrightarrow[]{\Spec^h(f)} \Spec^h(\pi_*(R)).
     \]
      We simply write  $\varphi$ when there is no confusion.
\end{notation}

Let us establish some terminology following \cite[Section 2.1]{BCHV_Noetherian}.

\begin{definition} 
    Let $f\colon R\to S$ be a map of commutative Noetherian ring spectra. We say that $f$ satisfies \textit{simple Quillen lifting} if for any $R$--module $M$, the following equalities hold:
    \[
     \varphi^{-1} \varphi \supp_S(f^\ast M)= \supp_S (f^\ast M), \mbox{ and }\]
    \[  \varphi^{-1}  \varphi \cosupp_S(f^! M)= \cosupp_S (f^! M).
    \]
\end{definition}

Our interest in simple Quillen lifting stems from \cite[Theorem 2.14]{BCHV_Noetherian}, which relates this condition to stratification. We now give a criterion for detecting simple Quillen lifting.

\begin{proposition}\label{prop-simple-quillen}
     Let $I$ be a finite set and $f=\prod f_i\colon R \to S = \prod_{i \in I}S_i$ be a map of commutative Noetherian ring spectra with $\Mod_{S_i}$ stratified by $\pi_\ast S_i$, for all $i \in I$.  Assume that:
       \begin{itemize}
           \item For any $\pp\in \Spec^h(\pi_\ast S_i)$ and $\qq\in \Spec^h(\pi_\ast S_j)$ such that $\varphi_{f}(\pp)=\varphi_{f}(\qq)$, there are $k\in I$ and maps of ring spectra $f_{k,i}\colon S_i\to S_k$ and $f_{k,j}\colon S_j\to S_k$ (compatible with $R$) with $\varphi_{f_{k,i}}^{-1}(\mathfrak{p})\cap \varphi_{f_{k,j}}^{-1}(\mathfrak{q})\neq \varnothing$.
       \end{itemize}
Then $f$ satisfies simple Quillen lifting. 
\end{proposition}

\begin{proof}
   This follows verbatim from the proof of \cite[Proposition 2.17]{BCHV_Noetherian} by noting that the Avrunin–Scott identities hold for $f_{k,i}\colon S_i\to S_k$ with the assumption that $\Mod_{S_i}$ is stratified (by \cite[Corollary 14.19]{BCHS}) for all $i$. We include the proof for completeness.

   The inclusion $\supp_S (f^*(M))\subseteq \varphi_{f}^{-1}\varphi_{f}(f^*(M))$ is always satisfied. We can also assume that $\supp_S (f^*(M))$ is non-empty, otherwise the equality is also satisfied. We need to proof that given $\qq,\mathfrak{r}\in \Spec^h(\pi_*(S))$ such that $\varphi_{f}(\qq)=\varphi_{f}(\mathfrak{r})$ with $\mathfrak{r}\in \supp_{S_i}(f_i^*M)$ for some $i$, there is $j\in I$ such that $\mathfrak{q}\in \supp_{S_j}(f_j^*M)$.
   
    Assume $\mathfrak{q}\in \Spec^h(\pi_*(S_j))$. By hypothesis, there are $k\in I$ and maps of ring spectra $f_{k,i}\colon S_i\to S_k$ and $f_{k,j}\colon S_j\to S_k$ (compatible with $R$, that is, $f_{k,i}\circ f_i\simeq f_k$ and $f_{k,j}\circ f_j\simeq f_k$) with $\varphi_{f_{k,i}}^{-1}(\mathfrak{r})\cap \varphi_{f_{k,j}}^{-1}(\mathfrak{q})\neq \varnothing$. Let $\pp\in \varphi_{f_{k,i}}^{-1}(\mathfrak{r})\cap \varphi_{f_{k,j}}^{-1}(\mathfrak{q})\subseteq \Spec^h(\pi_*(S_k))$.

Since the categories $\Mod_{S_i}$ are stratified for all $i\in I$, we have Avrunin-Scott identities $\varphi_{f_{k,i}}^{-1}(\supp(f_i^*M))=\supp((f_{k,i})^*(f_i^*M))$ and the same holds for $f_{k,j}$; see \cite[Proposition 3.14]{BCHV}. Then,  
\[
 \varphi_{f_{k,i}}^{-1}(\supp(f_i^*M))=\supp((f_{k,i})^*(f_i^*M))=\supp((f_{k,j})^*(f_j^*M))= \varphi_{f_{k,j}}^{-1}(\supp(f_j^*M))
\]
where the middle equality holds since $(f_{k,i})^*(f_i^*M)=f_k^*M=(f_{k,j})^*(f_j^*M)$.

Since $\mathfrak{r}\in \supp_{S_i}(f_i^*M)$, then $\pp \in \varphi_{f_{k,i}}^{-1}(\supp(f_i^*M))=\varphi_{f_{k,j}}^{-1}(\supp(f_j^*M))$. Finally, $\mathfrak{q}=\varphi_{f_{k,j}}(\pp) \in \supp(f_j^*M)$.

The same argument is used with cosupport.
\end{proof}

\begin{remark}\label{remark:colim of Spech}
    Let $I$ be the set of objects of a finite category $\mathcal{I}$, and $\S \colon \mathcal{I}\to \textrm{CAlg(Sp)}$ be a functor. For each $i\in I$, we have a commutative ring spectrum $\S(i)\eqqcolon S_i$. Suppose we have a map of commutative ring spectra $f\colon R\to \lim_{\mathcal I} \mathcal{S}\subseteq \prod_{i\in I} S_i$. This data induces 
    \[
    \varphi_{f}\colon \bigsqcup_{i\in I} \Spec^h(\pi_*(S_i))\stackrel{\pi}{\to}\textrm{colim}_{\mathcal I} \Spec^h(\pi_*(S_i))\to \Spec^h(\pi_*(R)).
    \] 
\end{remark}

\begin{corollary}\label{colim-simple-quillen}
    Let $I$ be a finite set and $f=\prod f_i\colon R \to S = \prod_{i \in I}S_i$ be a map of commutative Noetherian ring spectra, and  $\Mod_{S_i}$ cohomologically stratified by $\pi_\ast S_i$, for all $i \in I$.  Under the assumptions on \cref{remark:colim of Spech}, if 
    \[
    \textrm{colim}_{\mathcal I} \Spec^h(\pi_*(S_i))\to \Spec^h(\pi_\ast (R))
    \]
    is injective, then  $f$ satisfies simple Quillen lifting.
\end{corollary}

\begin{proof}
    Let $\mathfrak{r}\in \Spec^h(\pi_\ast S)$ and $\qq\in \Spec^h(\pi_\ast S)$ such that $\varphi_{f}(\mathfrak{r})=\varphi_{f}(\qq)$. Then $\pi(\pp)=\pi(\qq)$, that is, both primes are identified in the colimit. Precisely, that means that they are connected by a zig-zag
    
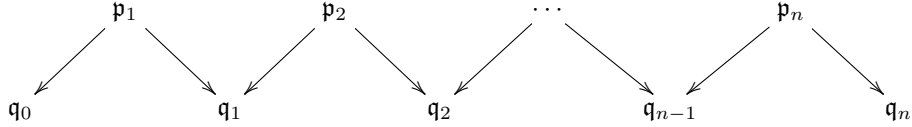
\begin{center}
\xymatrix{
 & \mathfrak{p}_1 \ar@{->}[ld] \ar@{->}[rd] &  & \mathfrak{p}_2 \ar@{->}[ld] \ar@{->}[rd] &  & \cdots \ar@{->}[ld] \ar@{->}[rd] &  & \mathfrak{p}_n \ar@{->}[ld] \ar@{->}[rd] &  \\
\mathfrak{q}_0 &  & \mathfrak{q}_1 &  & \mathfrak{q}_2 &  & \mathfrak{q}_{n-1} &  & \mathfrak{q}_n
}
  \end{center}  
  with $\mathfrak{r}=\mathfrak{q}_0$, and $\mathfrak{q}_{n}=\mathfrak{q}$, $\mathfrak{q_{i}}\in \Spec^h(\pi_*(S_{k_i}))$, and $\mathfrak{p_{i}}\in \Spec^h(\pi_*(S_{l_i}))$, and morphisms induced from the functor $\S$. 

  The goal is to check that simple Quillen lifting is satisfied. That is, given $R$--module $M$, the following equalities hold:
    \[
    \varphi_{f}^{-1} \varphi_{f} \supp_S(f^\ast M)= \supp_S (f^\ast M), \mbox{ and }\]
    \[ \varphi_{f}^{-1} \varphi_{f} \cosupp_S(f^! M)= \cosupp_S (f^! M).
    \]
    As in the proof of \cref{prop-simple-quillen}, it reduce to check that if $\mathfrak{r}\in \supp(f^*M)$ then the same holds for $\mathfrak{q}$.
    Note that if $\varphi_{f}(\mathfrak{q}_0)=\mathfrak{p}$, then the same holds for any $\mathfrak{q}_i$ in the diagram. By \cref{prop-simple-quillen}, we have that if $\mathfrak{r}\in \supp(f_{k_0}^*M)$ then $\mathfrak{q}_1\in \supp(f_{k_1}^*M)$. Applying \cref{prop-simple-quillen} again, we finally have that $\mathfrak{q}=\mathfrak{q}_n\in \supp(f_{k_n}^*M)$. The same argument is applied to cosupport.
\end{proof}

In particular, we can apply the previous result to deduce simple Quillen lifting for modules over $C^\ast(BG;k)$. We need some preparation. 

\begin{remark}
\label{lemma-Orb-finite}
    Let $G$ be a group admitting a finite model for $\underline{E}G$, and $k$ be a commutative Noetherian ring. Let $\mathcal{A}_{\Fin}(G)$ denote the category of finite subgroups of $G$ with morphisms being group homomorphisms induced by inclusions and conjugations in $G$ (that is, as a subcategory of the category of groups and group morphisms). This category is also known as the fusion-orbit category. 
    
    Note that $\mathcal{A}_{\Fin}(G)$ is equivalent to a finite category since $G$ has finitely many conjugacy classes of finite subgroups, and there is a finite set of morphisms between any two objects of $\mathcal{A}_{\Fin}(G)$. 
    
Moreover, there is a functor $\mathcal{O}_{\Fin}(G)\to \mathcal{A}_{\Fin}(G)$ which, on objects, sends the coset $G/F$ to $F$, and, on morphisms $[g]\colon G/F\to G/F'$, sends $[g]$ to the corresponding subconjugation $c_g\colon F \to {}^gF\leq F'$. Note that this functor induces an isomorphism
  \[
  \Hom_{\mathcal{O}_{\Fin}(G)}(G/F,G/F')/C_G(F)\cong \Hom_{\mathcal{A}_{\Fin}(G)}(F,F').
  \]
  Moreover, since morphisms in the orbit category given by elements in $C_G(F)$ induce the identity in cohomology with trivial coefficients, this functor induces an isomorphism  
    \[
    \lim_{{\mathcal{A}_{\Fin}}(G)^{\textrm{op}}}H^\ast(F;k) \stackrel{\cong}{\to} \lim_{\mathcal{O}_{\Fin}(G)^{\textrm{op}}}H^\ast(F;k).
    \]
\end{remark}

\begin{corollary}\label{coro-simplequillen-G}
   Let $G$ be a group admitting a finite model for $\underline{E}G$, and $k$ be a commutative Noetherian ring. Consider the ring map
    \[
    \mathrm{Ind}_{\Fin(G)/G}\colon {C^\ast(BG;k)}\to \prod_{(F)\in \Fin(G)/G}{C^\ast(BF;k)}.
    \]
    Then $\mathrm{Ind}_{\Fin(G)/G}$ satisfies simple Quillen lifting. 
\end{corollary}

\begin{proof}
   We will verify the conditions of \cref{colim-simple-quillen}. Recall that $C^\ast(BG;k)$ is Noetherian by \cref{thm-cfg}. Now,   
   by \cite[Lemma 8.11]{Qui71}, \cref{prop-fg-limitF} and \cref{lemma-Orb-finite}, we obtain that 
   \[
   \Spec^h\left(\lim_{{\mathcal{A}_{\Fin}}(G)^{\textrm{op}}}H^\ast(F;k)\right)\cong \underset{\mathcal{A}_{\Fin}(G)^{\textrm{op}}}{\mathrm{colim}} \Spec^h(H^\ast(F;k)). 
   \] 
   Moreover, \cref{f-iso induces homeo} applied to \cref{lemma-Orb-finite} and \cref{prop-F-i} gives us a homeomorphism 
   \[
   \Spec^h(H^\ast(G;k))\cong \underset{\mathcal{A}_{\Fin}(G)^{\textrm{op}}}{\mathrm{colim}} \Spec^h(H^\ast(F;k))
   \] 
Note that this homeomorphism is induced by $\mathrm{Ind}_{\Fin(G)/G}$. Thus, we are under the hypotheses of \cref{colim-simple-quillen}, and the result follows.
\end{proof}


The last ingredient that we need in order to prove cohomological stratification of the module category $\Mod_{C^\ast(BG;k)}$ is the following.

\begin{recollection}\label{rec-stratifaciton-finitecase}
    Let $G$ be a finite group, and $k$ be a commutative Noetherian ring. Then we have  
    \[
    \Mod_{C^\ast(BG;k)}\simeq \mathrm{Loc}_{ \Mod_{{\mathrm{Sp}_G}}(\underline{k}_G)}(\underline{k}_G) 
    \]
   by Morita theory; see  \cref{rec-Morita theory for R_G}. On the other hand, the category $\Mod_{{\mathrm{Sp}_G}}(\underline{k}_G)$ corresponds to the so called category of representations $\mathrm{Rep}(G,k)$ which was introduced in \cite[Section 5]{barthel2025lattices}. In particular, this category is cohomologically stratified  by $H^\ast(G;k)$ \cite[Theorem 9.1]{barthel2025lattices} (see also \cite[Section 5]{Gom25}). Finally, by \cite{benson2011localising}, we obtain that $\Mod_{C^\ast(BG;k)}$ is cohomological stratified by $H^\ast(G;k)$.
\end{recollection}

We now are ready to prove the main result of this section.  

\begin{theorem}\label{thm-cohomologica-stra}
Let $k$ be a commutative Noetherian ring and $G$ be a group admitting a finite model for $\underline{E}G$. Then $\Mod_{C^\ast(BG;k)}$ is cohomologically stratified by the canonical action of $H^\ast(BG;k)$.
\end{theorem}

\begin{proof}
 For a finite group $F$, the module category $\Mod_{C^\ast(BF;k)}$ is cohomologically stratified by $H^\ast(F;k)$; see \cref{rec-stratifaciton-finitecase}. On the other hand, by \cref{lemma-Orb-finite} applied to \cref{coro-simplequillen-G} we get that  
  \[
    \mathrm{Ind}_{\Fin(G)/G}\colon {C^\ast(BG;k)}\to \prod_{(F)\in \Fin(G)/G}{C^\ast(BF;k)}
    \]
    is a map of commutative Noetherian ring spectra that 
    satisfies simple Quillen lifting. Moreover, $\mathrm{Ind}_{\Fin(G)/G}$ is biconservative by \cref{thm-weakly}. We conclude the result by \cite[Theorem 2.14]{BCHV_Noetherian}.
\end{proof}

\subsection{Further consequences of \cref{thm-weakly}}\label{sec-homological}
In this section, we record two direct consequences of our weak descendability result, \cref{thm-weakly}: one concerning the topology of the tt-spectrum of $\Mod_{C^\ast(BG;R)}^\omega$, and another concerning the classification of localizing subcategories of the module category $\Mod_{C^\ast(BG;R)}$ in terms of the so-called homological spectrum.

The first consequence is the following.

\begin{corollary}\label{prop-balmerspectrum-noe}
    Let $G$ be a discrete group with a stable finite-dimensional model $X$ for $\underline{E}G$, and $R$ be a commutative ring spectrum. Assume that $R$ has a \textit{$G$-finite derived defect base} $\mathcal{F}$. If $\Mod_{C^\ast(BH;R)}^\omega$ has Noetherian Balmer spectrum for each $H\in \mathcal{F}$, then so does $\Mod_{C^\ast(BG;R)}^\omega$.
\end{corollary}

\begin{proof}
  By \cref{thm-weakly}, the geometric functor 
  \[
  \mathrm{Ind}_\mathcal{F}\colon \Mod_{C^\ast(BG;R)}\xrightarrow[]{}\prod_{H\in \mathcal{F}}\Mod_{C^\ast(BH;R)}
  \]
  is conservative. A standard argument shows that $\mathrm{Ind}_\mathcal{F}$ detects $\otimes_R$-nilpotent morphisms whose domain is compact. Then, by \cite{balmer2018surjectivity} (cf. \cite{BCHS}) we get that 
\[
\mathrm{Spc}(  \mathrm{Ind}_\mathcal{F})\colon\bigsqcup_{E\in \mathcal{F}} \Spc(\Mod_{C^\ast(BE;R)}^\omega) \to \Spc(\Mod_{C^\ast(BG;R)}^\omega)
\]
is surjective. Since the left hand side is Noetherian, then so is its image under $\mathrm{Spc}(  \mathrm{Ind}_\mathcal{F})$.
\end{proof}

\begin{example}
     The previous result applies for instance to $R=E$, the Lubin-Tate theory of height $n$ and a prime $p$, and any group with a finite model for $\underline{E}G$. 
\end{example}

 For the second consequence, recall that \emph{homological stratification} amounts to a bijection between the localizing ideals of a rigidly-compactly generated tensor-triangulated category $\T$ and the subsets of the homological spectrum $\mathrm{Spc}^h(\T^\omega)$.

\begin{corollary}\label{coro-homologicalstra}
         Let $G$ be a discrete group with a {stable} finite-dimensional model $X$ for $\underline{E}G$, and $R$ be a commutative ring spectrum. Assume that $R$ has a \textit{$G$-finite derived defect base} $\mathcal{F}$, and that $\Mod_{C^\ast(BH;R)}$ is homologically stratified for each $H\in \mathcal{F}$.  
         Then the category $\Mod_{C^\ast(BG;R)}$ is homologically stratified.
\end{corollary}

\begin{proof}
   By  \cref{coro:chouinard in the field case} we know that the geometric functor $\mathrm{Ind}_\mathcal{F}$ is weakly descendable. The claim then follows from  \cite[Theorem 7.23]{barthel2024homological}. 
\end{proof}

\begin{remark}
One might hope to use \cref{prop-balmerspectrum-noe} to promote homological stratification, as in \cref{coro-homologicalstra}, to triangular stratification (see \cite{barthel2024homological} for details on this notion, which builds on the work of Stevenson \cite{Ste13}), or even to cohomological stratification. For instance, when $R=E$ is the Lubin--Tate theory of height $n$ at a prime $p$, cohomological stratification is already known for finite groups; see \cite[Section 7]{BCHNP}. In light of \cite{barthel2024homological} together with the previous results, triangular stratification is equivalent to the so-called Nerves of Steel conjecture. We do not pursue this direction here.
\end{remark}

\bibliographystyle{alpha}
\bibliography{bibfile} 

\end{document}